\documentclass[12pt]{amsart}
\usepackage{a4wide}
\usepackage{amssymb}
\usepackage{amsthm}
\usepackage{amsmath}
\usepackage{amscd}
\usepackage{verbatim}
\usepackage[all]{xy}

\numberwithin{equation}{section}

\theoremstyle{plain}
\newtheorem{theorem}{Theorem}[section]
\newtheorem{corollary}[theorem]{Corollary}
\newtheorem{lemma}[theorem]{Lemma}
\newtheorem{proposition}[theorem]{Proposition}

\newtheorem{conjecture}[theorem]{Conjecture}
\theoremstyle{definition}

\newtheorem{remark}[theorem]{Remark}

\theoremstyle{remark}

\newcommand{\OO}{\mathcal O}
\newcommand{\A}{\mathbb{A}}
\newcommand{\R}{\mathbb{R}}
\newcommand{\G}{\mathbb{G}}
\newcommand{\Q}{\mathbb{Q}}
\newcommand{\Z}{\mathbb{Z}}

\newcommand{\C}{\mathbb{C}}

\renewcommand{\H}{\mathbb{H}}

\newcommand{\D}{\mathbb{D}}

\newcommand{\zxz}[4]{\begin{pmatrix} #1 & #2 \\ #3 & #4 \end{pmatrix}}
\newcommand{\abcd}{\zxz{a}{b}{c}{d}}

\newcommand{\kzxz}[4]{\left(\begin{smallmatrix} #1 & #2 \\ #3 & #4\end{smallmatrix}\right) }

\newcommand{\calE}{\mathcal{E}}
\newcommand{\calF}{\mathcal{F}}

\newcommand{\calO}{\mathcal{O}}

\newcommand{\calT}{\mathcal{T}}

\newcommand{\calX}{\mathcal{X}}
\newcommand{\calZ}{\mathcal{Z}}

\newcommand{\eps}{\varepsilon}
\newcommand{\bs}{\backslash}
\newcommand{\norm}{\operatorname{N}}

\newcommand{\vol}{\operatorname{vol}}
\newcommand{\tr}{\operatorname{tr}}

\newcommand{\GSpin}{\operatorname{GSpin}}
\newcommand{\CT}{\operatorname{CT}}

\newcommand{\Orth}{\operatorname{O}}

\newcommand{\Hom}{\operatorname{Hom}}
\newcommand{\Aut}{\operatorname{Aut}}

\newcommand{\Spec}{\operatorname{Spec}}

\newcommand{\End}{\operatorname{End}}

\newcommand{\sig}{\operatorname{sig}}

\newcommand{\Pet}{\text{\rm Pet}}

\newcommand{\GL}{\operatorname{GL}}
\newcommand{\SO}{\operatorname{SO}}

\newcommand{\supp}{\operatorname{supp}}

\newcommand{\dv}{\operatorname{div}}

\newcommand{\scal}{\text{\rm sc}}

\newcommand{\cha}{\operatorname{char}}
\newcommand{\diag}{\operatorname{diag}}
\newcommand{\ord}{\operatorname{ord}}

\newcommand{\Gspin}{\operatorname{GSpin}}

\newcommand{\ff}{\hbox{if }}
\newcommand{\SL}{\operatorname{SL}}
\newcommand{\Res}{\operatorname{Res}}
\newcommand{\CH}{\operatorname{CH}}
\newcommand{\reg}{\operatorname{reg}}

\newcommand{\CM}{\operatorname{CM}}

\newcommand{\s}{\sigma}
\newcommand{\RR}{\text{\rm Res}}   
\newcommand{\ttt}{\otimes}
\newcommand{\lra}{\longrightarrow}
\newcommand{\ph}{\varphi}

\newcommand{\action}{\bullet}
\newcommand{\reflex}{\sharp}

\newcommand{\CCM}{\mathcal C \mathcal M}
\newcommand{\st}{\operatorname{st}}
\newcommand{\PF}{\operatorname{PF}}
\newcommand{\bA}{\bold A}

\newcommand{\isoarrow}{\ {\overset{\sim}{\longrightarrow}}\ }   

\newcommand{\LL}{\mathcal L} 

\newcommand{\DD}{^\Delta}    

\begin{document}

\baselineskip =16pt

\title[Faltings heights of big CM cycles and derivatives of  $L$-functions]{Faltings heights of big CM cycles\\ and derivatives of  $L$-functions}

\author[Jan H.~Bruinier, Stephen S. Kudla,  and Tonghai Yang]{Jan
Hendrik Bruinier, Stephen S. Kudla, and Tonghai Yang}
\address{Fachbereich Mathematik,
Technische Universit\"at Darmstadt, Schlossgartenstrasse 7, D--64289
Darmstadt, Germany}
\email{bruinier@mathematik.tu-darmstadt.de}
\address{Department of Mathematics, University of Toronto, Toronto, Ontario, Canada M5S
2E4} \email{skudla@math.utoronto.ca}
\address{Department of Mathematics, University of Wisconsin Madison, Van Vleck Hall, Madison, WI 53706, USA}
\email{thyang@math.wisc.edu} 

\begin{abstract}
  We give a formula for the values of automorphic Green functions on
  the special rational 0-cycles (big CM points) attached to certain
  maximal tori in the Shimura varieties associated to rational
  quadratic spaces of signature $(2d,2)$. Our approach depends on the
  fact that the Green functions in question are constructed as
  regularized theta lifts of harmonic weak Mass forms, and it involves
  the Siegel-Weil formula and the central derivatives of incoherent
  Eisenstein series for totally real fields.  In the case of a weakly
  holomorphic form, the formula is an explicit combination of
  quantities obtained from the Fourier coefficients of the central
  derivative of the incoherent Eisenstein series. In the case of a
  general harmonic weak Maass form, there is an additional term given
  by the central derivative of a Rankin-Selberg type convolution.
\end{abstract}

\thanks{The first author is partially supported by DFG grant BR-2163/2-1. 
The third author is partially supported by grants NSF DMS-0855901
and NSFC-10628103.}


\date{\today}
\maketitle

\section{\bf Introduction} \label{sect1}

In 1985, Gross and Zagier discovered a beautiful factorization formula
for singular moduli \cite{GZSingular}. This has inspired a lot of
interesting work, including Dorman's generalization to odd discrinants
\cite{Do}, Elkies's examples on Shimura curves \cite{El} and Lauter's
conjecture on the Igusa $j$-invariants (\cite{GL}, \cite{Yam=1},
\cite{YaGeneral}), among others. In his thesis, Schofer \cite{Scho}
proved a much more general factorization formula for the `small' CM
values of Borcherds modular functions on a Shimura variety of
orthogonal type via regularized theta liftings. The proof is very
natural and is based on a method introduced in \cite{KuIntegral}.  Two
of the authors adapted the same idea to study the `small' CM values of
automorphic Green functions
and discovered a direct link between the CM value and the central
derivative of a certain Rankin-Selberg $L$-function. This direct link
is used to give a different proof of the well-known Gross-Zagier
formula \cite{BY2}. Here `small' means that the CM cycles are
associated to quadratic imaginary quadratic fields.  On the other
hand, the two authors also extended Gross and Zagier's factorization
formula, using a method close to Gross and Zagier's original idea, to
`big' CM values of some Hilbert modular functions on a Hilbert modular
surface. Here `big' means that the CM cycle is associated to a maximal
torus of the reductive group giving the Hilbert modular surface.

A motivating question for this paper is whether this `big' CM value
result can also be derived using the regularized theta lifting method
in \cite{Scho} and \cite{BY2}, which is more natural and simpler.
While the small CM cycles are constructed systematically and
associated to rational negative two planes in the quadratic space
defining the Shimura variety, no big CM cycles are constructed this
way. In Section \ref{sect2}, we describe a way to construct big CM
cycles in some special Shimura varieties (including Hilbert modular
surfaces), and study their Galois conjugates. In Sections
\ref{sect3}--\ref{sect5}, we extend the CM value result in \cite{BY2}
to this situation. In Section \ref{Hilbert}, we restrict to the
special case of Hilbert modular surfaces and give a new proof of the
main results in \cite{BY1} and a generalization.  Actually, to get the
CM cycles in \cite{BY1} from this construction is not straightforward
and quite interesting. An arithmetic application is given at the end
of Section \ref{Hilbert}.  We now describe this work in more detail.

Let $(V, Q_V)$ be a rational quadratic space of signature $(2d, 2)$
for some positive integer $d \ge 1$.  Let $G=\Gspin(V)$ and let $K
\subset G(\hat{\Q})$ be a compact open subgroup\footnote{We write
  $\hat\Q = \Q\otimes_\Z{\hat\Z}$ for the finite ad\`eles of $\Q$,
  where $\hat\Z = \varprojlim_{n}\Z/n\Z$.}. Let $\mathbb D$ be the
associated Hermitian domain of oriented negative $2$-planes in $V(\R)=
V\otimes_\Q \R$, and let
\begin{equation}
X_K = G(\Q) \backslash \big(\,\mathbb D \times G(\hat{\Q})/K\,\big)
\end{equation}
be the associated Shimura variety which has a canonical model over
$\Q$.  Assume that there is a totally really number field $F$ of
degree $d+1$ and a two-dimensional $F$-quadratic space $(W, Q_W)$ of
signature
$$
\sig(W) =((0, 2), (2, 0), \cdots, (2,0))
$$
with respect to the $d+1$ embeddings $\{ \sigma_j\}_{j=0}^d$ such
that
$$
V= \Res_{F/\Q} W,  \quad Q_V(x) = \tr_{F/\Q} Q_W(x).
$$
Then there is an orthogonal direct sum decomposition
$$
V(\R) =\oplus_j W_{\sigma_j},  \quad W_{\sigma_j} =W\otimes_{F,
\sigma_j} \R.
$$
The negative $2$-plane $W_{\sigma_0}$ gives rise to two points
$z_0^\pm$ in $\mathbb D$. Let $T$ be the preimage of $\Res_{F/\Q}
\SO(W) \subset \SO(V)$ in $G$. Then $T$ is a maximal torus associated
to the CM number field $E =F(\sqrt{-\det W})$, and we obtain a `big'
CM cycle in $X_K$:
$$
Z(W, z_0^\pm) = T(\Q) \backslash \big(\,\{ z_0^\pm\} \times T(\hat{\Q})/K_T\,\big),
$$
where $K_T = T(\hat{\Q}) \cap K$. The CM cycle $Z(W, z_0^\pm)$ is
defined over $F$, and the formal sum $Z(W)$ of all its Galois
conjugates as a $0$-cycle in $X_K$ is defined over $\Q$. We refer to
Section~\ref{sect2} for details.

Let $L$ be an even integral lattice in $V$ such that $K$ preserves $L$
and acts trivially on $L'/L$, where $L'$ is the dual lattice.  Let
$S_L$ be the space of locally constant functions on $\hat V=V\otimes
\hat{\Q}$ which are $\hat L$-invariant and have support in $\hat{L}'$,
and let $\rho_L$ be the associated `Weil representation' of
$\SL_2(\Z)$ on it.  For each harmonic weak Maass form $f \in H_{1-d,
  \bar{\rho}_L}$, there is a corresponding special divisor $Z(f)$
(determined by the principal part of $f$) and an automorphic Green
function $\Phi(\cdot, f)$ which is constructed in \cite{BF} as a
regularized theta lift of $f$ (see Section \ref{sect3}). On the other
hand, associated to $L$, there is also an incoherent (vector valued
and normalized) Hilbert Eisenstein series $E^*(\vec \tau, s, L, \bold
1)$ of parallel weight $1$ (see Section \ref{sect4}) such that its
diagonal restriction to $\Q$ is a weight $d+1$ non-holomorphic modular
form with representation $\rho_L$.  Let $\mathcal E(\tau, L)$ be the
`holomorphic part' of $E^{*, \prime}(\tau^\Delta, 0, L, \bold 1)$,
where, for $\tau \in \mathbb H$, we put $\tau^\Delta=(\tau, \cdots,
\tau) \in \mathbb H^{d+1}$. Finally define the generalized
Rankin-Selberg $L$-function
 \begin{equation}
\LL(s, \xi(f), L) = \langle E^*(\tau^\Delta, s, L, \bold 1),
\xi(f)\rangle_{\Pet}
 \end{equation}
 to be the Petersson inner product of the pullback of the Eisenstein
 series and the holomorphic cusp form $\xi(f)$ of weight $d+1$, given
 by the differential operator
 $\xi(f)=2iv^{1-d}\overline{\frac{\partial f}{\partial \bar \tau}}$.
In Section \ref{sect5}, we prove the following general formula, which
is similar to that in \cite{Scho} and \cite[Theorem
\ref{theo5.1}]{BY2}.

\begin{theorem} 
\label{theo1.1}  
Let the notation be as above. Then
\begin{equation}
\label{eq:1.1}
\Phi(Z(W), f) = \frac{\deg Z(W, z_0^\pm)}{\Lambda(0, \chi_{E/F})}
\left( \CT[\langle f^+(\tau), \mathcal E(\tau, L)\rangle]-
\LL'(0,\xi(f), L)\right).
\end{equation}
Here $\chi_{E/F}$ is the quadratic Hecke character of $F$ associated
to $E/F$, $f^+$ is the holomorphic part of $f$, and $\CT[\langle
f^+(\tau), \mathcal E(\tau, L)\rangle]$ is the constant term  of
$$
\langle f^+(\tau), \mathcal E(\tau, L)\rangle =
\sum_{\mu \in L'/L} f^+(\tau, \mu) \,\mathcal E(\tau, L,\mu),
$$
where $f^+(\tau, \mu)$ is the $\mu$-component of $f^+$, and $\mathcal E(\tau, L,\mu)$ is the $\mu$-component of $\mathcal E(\tau, L)$.
\end{theorem}

In the special case that $f$ is weakly holomorphic, $\Phi(\cdot, f)$
is the Petersson norm of a meromorphic modular form $\Psi(\cdot,f)$ on
$X_K$ given by the Borcherds lift of $f$.  The second summand on the
right hand side of \eqref{eq:1.1} vanishes and the first summand gives
an explicit formula for the evaluation of $\Psi(\cdot,f)$ on the CM
cycle $Z(W)$.

Note that the first summand $\CT[\langle f^+(\tau), \mathcal E(\tau,
L)\rangle]$ is of arithmetic nature and this theorem suggests two
interesting conjectures about arithmetic intersection numbers and
Faltings heights of big CM cycles, see Conjectures \ref{conj5.2} and
\ref{conj5.3}.

Also note that, in contrast with the situation in \cite{BY2}, the
function $\LL(s, \xi(f),L)$ is not a standard Rankin-Selberg integral,
since it involves the pullback of a Hilbert modular Eisenstein series.
We expect that it is related to a Langlands $L$-function for the group
$G$ and hope to pursue this idea in a subsequent paper.

To explain the Hilbert modular surface case in \cite{BY1}, let $E$ be
a non-biquadratic quartic CM number field with real quadratic subfield
$F=\Q(\sqrt D)$ with fundamental discriminant $D$.  Let $ \sigma $ be
the non-trivial Galois automorphism of $F$.  Let
$$ V:= \{ A \in M_2(F):\,  \sigma(A)
=A^\iota\} =\{ A =\kzxz {u} {b \sqrt D} {\frac{a}{\sqrt D}}
{\sigma(u)}:\,  u \in F, a, b\in \mathbb Q\}
$$
and let
$$
L =\{ A =\kzxz {u} {b \sqrt D} {\frac{a}{\sqrt D}} {\sigma(u)}:\, u
\in \OO_F , a ,  b  \in \Z\}.
$$
Here $A\mapsto A^\iota$ is the main involution of $M_2(F)$. The group
$$G(\Q)=\Gspin(V)(\Q) = \{ g \in \GL_2(F): \,\det g \in \Q^\times\}
$$
acts on $V$ via $g.A =\sigma( g) A g^{-1}$. The Shimura variety $X_K$
is a Hilbert modular surface, and, for suitable choice of $K$, is
isomorphic to $\SL_2(\OO_F) \backslash \mathbb H^2$. Now we describe
the CM cycle $\CM(E)$ in \cite{BY1}, the locus of abelian surfaces
over $\C$ with CM by $\OO_E$, as a formal sum of $Z(W)$'s. For a
principally polarized CM abelian surface $\bold A =(A, \kappa,
\lambda)$ of CM type $(\OO_E, \Sigma)$, let $M =H_1(A, \Z)$ with the
action of $\OO_E$ induced by $\kappa$ and the symplectic form
$\lambda$ induced by the polarization. Define the lattice
$$
L(\bold A)=\{ j \in \End(M):\,  j \circ \kappa(a) = \kappa(\sigma(a)) \circ j, \ a \in \OO_F,  \  j^* =j\}
$$
of special endomorphisms of $M$ with $\Z$-quadratic form $Q(j) =j^2$,
where $j^*$ is the `Rosati' involution induced by $\lambda$.  Let
$V(\bold A) = L(\bold A) \otimes \Q$. Then one can show that the rank
$4$ quadratic lattice $(L(\bold A), Q) \cong (L, \det)$ is independent
of the choice of $\bold A$.  On the other hand, let $\tilde E$ be the
reflex field of $(E, \Sigma)$, which is generated by the type norms
$\norm_\Sigma(r), r \in E$, defined in (\ref{eq:Typenorm}), and let
$\tilde F=\Q(\sqrt{\tilde D})$ be the real quadratic subfield of
$\tilde E$.  It turns out (\cite{HY}, see also Section \ref{Hilbert})
that $V(\bold A)$ has a natural $\tilde E$-vector space structure
together with an $\tilde F$-valued quadratic form $Q_{\bold A}$ such
that
$$
\norm_\Sigma(r) \action j = \kappa(r) \circ j \circ \kappa(\bar r)
$$
for any $r \in E$, and
$$
\tr_{\tilde F/\Q} Q_{\bold A}(j) = Q(j),  \quad j \in V(\bold A).
$$
Let $W(\bold A) =(V(\bold A), Q_\bold A)$ be the resulting
$2$-dimensional quadratic space over $\tilde F$.  Then the rational
torus associated to $W(\bold A)$ is
$$
T(R)=\{ r \in (R\otimes_\Q E)^\times:\,  r \bar r \in R^\times\}
$$
and its rational points $T(\Q)$ act on $W(\bold A)$ via $r\action j =
\frac{1}{r \bar r}\, \kappa(r) \circ j \circ \kappa(\bar r)$. However,
in general, different $\bold A$'s might give different $\tilde
F$-quadratic spaces and different incoherent Eisenstein series
$E^*(\vec \tau, s, L(\bold A), \bold 1)$. The CM cycle $\CM(E)$ is a
union of such CM cycles, and Theorem~\ref{theo1.1} gives a formula for
the CM value $\Phi(\CM(E), f)$ in terms of several incoherent
Eisenstein series. When $D \equiv 1 \mod 4$ is a prime, however, the
formula becomes simple and only one incoherent Eisenstein series is
involved, as in \cite{BY1} (see Theorem \ref{theo6.7}). We have the
following result (Theorem \ref{theo6.10}).

\begin{theorem}  
\label{theo1.2} 
Let $E$ be a CM quartic field with discriminant $D^2 \tilde D$ with $D \equiv 1 \mod 4$ prime
and $\tilde D \equiv 1 \mod 4$ square free and with real quadratic subfield $F=\Q(\sqrt D)$. Let $f \in H_{0, \bar{\rho}_L}$
as above. Then
$$
\Phi(\CM(E), f) = \frac{\deg(\CM(E))}{2 \Lambda(0, \chi)} \left(\CT[\langle f^+, \mathcal E(\tau, \tilde L)\rangle] - \LL'(0, \xi(f), \tilde L) \right).
$$
Here $\tilde L = \OO_{\tilde E}$ with $\tilde F$-quadratic form $\tilde Q(r) =-\frac{1}{\sqrt{\tilde D}}\, r \bar r$, and
$\Lambda(s, \chi)$ is the complete $L$-function of the quadratic Hecke character  $\chi$ of $F$ associated to $E/F$ defined in  (\ref{eq:L-series}).
\end{theorem}

We expect the factor $\frac{\deg(\CM(E))}{2 \Lambda(0, \chi)}$ to be
$1$ and prove it in some special cases in Section \ref{Hilbert}. We
also give a scalar modular form version of this theorem and arithmetic
applications in Section \ref{Hilbert}. In particular, we have the
following result.

\begin{theorem}  
\label{theo1.3} 
Assume that $E$ is a quartic CM number field with absolute
discriminant $d_E= D^2 \tilde D$ and real quadratic subfield
$F=\Q(\sqrt D)$ such that $D \equiv 1 \mod 4$ is prime and $\tilde D
\equiv 1 \mod 4$ is square-free. Assume further that
$$
\OO_E =\OO_F + \OO_F \frac{w+\sqrt\Delta}2
$$
is free over $\OO_F$, where $w, \Delta \in \OO_F$. Let $\mathcal X$ be
a regular toroidal compactification of the moduli stack of principally
polarized abelian surfaces with real multiplication by $\OO_F$,
\cite{rapoport.HB}, \cite{deligne-pappas}.  For any $f \in H_{0,
  \bar{\rho}_L}$ with $c^+(0, 0)=0$, let $\mathcal Z(f)$ be the
closure of $Z(f)$ in $\mathcal X$ and let $\hat {\mathcal Z}(f) =
(\mathcal Z(f), \Phi(\cdot, f))\in \widehat{\CH}^1(\mathcal X)_\C$.
Let $\CCM(E)$ be the moduli stack of principally polarized abelian
surfaces with CM by $\OO_E$. Then
$$
\langle \hat{ \mathcal Z}(f),  \CCM(E) \rangle_{\text{\rm Fal}} =
 -\frac{1}4  \LL'(0, \xi(f),
\tilde L).
$$
\end{theorem}

The idea of constructing big CM cycles was communicated to one of the
authors (T.Y.)  a couple of years ago by Eyal Goren in a private
conversation.  We thank him for sharing his idea. A slightly more
general type of CM point is discussed in \cite[Section 5]{shimura},
and our result (Theorem~\ref{theo1.1}) can undoubtedly be extended to
that case.

It is interesting to note that the Shimura variety $\text{\rm
  Sh}(G,\mathbb D)$ attached to $G= \GSpin(V)$ is of PEL-type only for
small values of $d$ where accidental isomorphisms occur. In these
cases, the moduli theoretic interpretation of the $0$-cycles defined
in Section 2 is slightly subtle.  Thus, for example, as shown in
Section 6, in the Hilbert modular surface case, the $0$-cycle
associated to abelian surfaces with CM by a non-biquadratic quartic CM
field $E/F$ is a union of the $0$-cycles constructed in Section 3 for
the reflex field $\tilde E/\tilde F$.

The second author would like to thank the Department of Mathematics at
the University Paris-Sud at Orsay, for their hospitality and
stimulating working environment during the month of May 2010. The
third author thanks the AMSS, the Morningside Center of Mathematics at
Beijing, and the Mathematical Science Center at Tsinghua University
for providing him excellent working conditions during his summer
visits to these institutes.

\section{\bf The  Shimura variety and its special points}
\label{sect2}

As in Section \ref{sect1}, let $F$ be a totally real number field of degree $d+1$ over $\Q$ with embeddings $\{\s_j\}_{j=0}^d$  into
$\R$.  Let $W$, $(\ ,\ )_W$ be a quadratic space over $F$ of dimension $2$ with signature
$$\sig(W) = ((0,2), (2,0), \dots, (2,0)).$$
Let $V = \RR_{F/\Q} W$ be the underlying rational vector space with bilinear  form
$(x,y)_V = \tr_{F/\Q}(x,y)_W$. There is an orthogonal direct sum
\begin{equation}\label{real.decompo}
V\ttt_\Q\R = \oplus_j   W_{\s_j}
\end{equation}
of real quadratic spaces where $W_{\s_j} = W\ttt_{F,\s_j}\R$, and $\sig(V) = (2d,2)$.
Let $G= \GSpin(V)$. Then there is a homomorphism
\begin{equation}\label{GSpinhom}
\RR_{F/\Q} \GSpin(W)\lra G
\end{equation}
of algebraic groups over $\Q$ which, on real points, gives the homomorphism
\begin{equation}\label{realpoints.hom}
\RR_{F/\Q} \GSpin(W)(\R) = \prod_j \GSpin(W_{\s_j}) \lra \GSpin(V\ttt_\Q\R) = G(\R),
\end{equation}
associated to the decomposition (\ref{real.decompo}).

\begin{lemma} Let $T$ be the inverse image
in $G$ of the subgroup $\RR_{F/\Q}\SO(W)$ of $\SO(V)$.  Then
$T$ is a maximal torus of $G$ and is the image of the homomorphism
(\ref{GSpinhom}).
\end{lemma}

Note that there is thus an exact sequence
\begin{equation}\label{exactT}
1\lra\G_m\lra T\lra \RR_{F/\Q}\SO(W)\lra 1
\end{equation}
of algebraic groups over $\Q$, where $\G_m$ is the kernel of the homomorphism
$\GSpin(V)\rightarrow \SO(V)$.

A more explicit description of $T$ can be given as follows. The even part $C^0_F(W)=E$
of the Clifford algebra of $W$ over $F$ is a CM field of degree $2d+2$ over $\Q$.
The odd part of the Clifford algebra $C^1_F(W)=W = E w_0$ is a one dimensional vector space
over $E$ with quadratic form $Q_W(a w_0 ) = \alpha\text{\rm N}_{E/F}(a)$, where $\alpha = Q_W(w_0) \in F^\times$
is an element with $\s_0(\alpha)<0$ and $\s_j(\alpha)>0$ for $j\ge 1$.
Then, on rational points, we have
$$
\begin{matrix}
\RR_{F/\Q}\GSpin(W)(\Q) &\longrightarrow&T(\Q)&\longrightarrow& \RR_{F/\Q}\SO(W)(\Q)\\
\noalign{\smallskip}
\Vert&{}&\Vert&{}&\Vert\\
\noalign{\smallskip}
E^\times&\longrightarrow&E^\times/F^1&\longrightarrow&E^\times/F^\times
\end{matrix}$$
where $ E^\times /F^\times\simeq E^1$, via $\beta \mapsto \beta/\bar{\beta}$ is the kernel of $\text{\rm N}_{E/F}$,
and $F^1$ is the kernel of $N_{F/\Q}$.

Fixing an identification $ \mathbb S=\Res_{\mathbb
C/\mathbb R} \mathbb G_m \simeq \Gspin(W_{\sigma_0})$,  we obtain a homomorphism
$h_0: \mathbb S \rightarrow G_\R$
of  algebraic groups over $\R$
corresponding to the inclusion in the first factor in (\ref{realpoints.hom}).
Let $\mathbb D$ be the $G(\mathbb R)$-conjugacy class of $h_0$. Let $\{e_0, f_0\}$ be a  standard
basis of $W_0 \subset V\ttt_\Q\R$. Then it is easy to check
$$
gh_0 g^{-1} \mapsto \mathbb R ge_0 + \mathbb R gf_0
$$
gives a bijection between $\mathbb D$ and the set of oriented
negative $2$-planes in $V\ttt_\Q\R$. We will not distinguish between the
two interpretations of $\mathbb D$. Note that the choice of orientation determined by $\{ e_0,f_0\}$
is equivalent to the choice of an extension of $\s_0$ to an embedding of $E$ into $\C$, which we also
denote by $\s_0$.

Let $K$ be a compact open
subgroup of $G(\hat{\mathbb Q})$, where $\hat F$ stands for the
finite adeles of a number field $F$. Let $X_K = \text{\rm Sh}(G,h_0)_K$ be the canonical model of the Shimura
variety over $\mathbb Q$ with
$$
X_K(\mathbb C) = G(\mathbb Q)  \backslash\big(\, \mathbb D \times
G(\hat{\mathbb Q})/K\,\big).
$$

%


By construction, the homomorphism $h_0$ factors through $T_\R$ and is fixed by conjugation by $T(\R)$,
so we have, for any $g\in G(\hat\Q)$,  a special $0$-cycle
in $X_K$ according to \cite[Page 325]{milne}
\begin{equation}\label{ZTpoints}
Z(T, h_0,g)_K = T(\mathbb Q)\backslash\big(\, \{h_0\}\times T(\hat{\mathbb
Q})/K^g_T\,\big) \rightarrow X_K, \qquad [h_0,t] \mapsto [h_0,tg]
\end{equation}
where $K^g_T=T(\hat{\mathbb Q}) \cap gKg^{-1}$.  Note that $K_T^g$ depends only on the image of
$g$ in $\SO(V)(\hat \Q)$.  We will usually drop the subscript $K$ and identify $Z(T, h_0,g)$
with its image in $X_K$, but every  point in $Z(T, h_0,g)$ is counted
with multiplicity $\frac{2}{w_{K, T,g}}$ and $w_{K, T,g}=\# (T(\mathbb
Q) \cap gKg^{-1}) $. In
particular, for a function $f $ on $X_K$, we have
\begin{equation}\label{weighted.sum}
f(Z(T, h_0,g)) =\frac{2}{w_{K, T,g}} \sum_{t\in T(\mathbb Q) \backslash
T(\A_f)/K^g_T} f(h_0, tg).
\end{equation}
When $g=1$, we will further abbreviate notation and write $Z(T,h_0)$ for $Z(T,h_0,1)$.

The $0$-cycle $Z(T, h_0)$ is defined over $\sigma_0(E)$, the reflex field of
$(T, h_0)$.  We next describe its Galois
conjugates $\tau(Z(T, h_0))$ for $\tau \in \Aut(\mathbb C/\mathbb
Q)$.

For $j\in \{ 0, \dots, d\}$, let $W(j)$ be the unique (up to isomorphism) quadratic space over $F$ such that $W(j)\otimes_F F_v$
and $W\otimes_F F_v$ are isometric for all finite place $v$ of $F$, and such that
\begin{equation}
\sig(W(j)) = ((2,0), \dots,(2,0), \underset{j}{(0,2)},(2,0) \dots, (2,0)).
\end{equation}
Note that, although the quadratic spaces $W=W(0)$ and $W(j)$ over
$F$ are not isomorphic for $j\ne0$, there is an isomorphism $C^0_F(W(j))\simeq
C^0_F(W)=E$ of their even Clifford algebras. Let
$V(j)=\RR_{F/\Q} W(j)$ with bilinear form $ (x,y)_{V(j)}=
\tr_{F/\mathbb Q} (x,y)_{W(j)}. $ The signature of $V(j)$ is
$(2d,2)$ and the quadratic spaces $V(j)$ and $V$ are isomorphic.  We
fix an isomorphism
\begin{equation}
V(j) \overset{\sim}{\longrightarrow} V
\end{equation}
 and hereafter identity $V(j)$ with $V$. Let $T(j)$ be the preimage of $\Res_{F/\mathbb Q} \SO(W(j))
\subset \SO(V)$ in $G$ and let $h_0(j): \mathbb S \rightarrow
G_\mathbb R$ be the homomorphism defined, as above, by an
identification of $\mathbb S$ with
$\Gspin(W(j)\otimes_{F,\sigma_j}\mathbb R)$. For $g\in G(\hat \Q)$, the analogue of the
construction above yields a special $0$-cycle $Z(T(j), h_0(j),g)$ on
$X_K$ defined over $\sigma_j(E)$.

We fix an $\hat F$-linear isometry
\begin{equation}
\mu_j: W(j)(\hat F) \overset{\sim}{\longrightarrow} W(\hat F).
\end{equation}
Noting that there are canonical identifications $W(j)(\hat F) = V(j)(\hat\Q)$
and $W(\hat F) = V( \hat\Q)$, and using the fixed identification of $V$ and $V(j)$,
there is a unique element $g_{j,0} \in \text{\rm O}(V)(\hat\Q)$ such that  the diagram
\begin{equation}\label{seesaw.relations}
\begin{matrix}
W(j)(\hat F)& \overset{\mu_j}{\longrightarrow} & W(\hat F)\\
\noalign{\smallskip}
||&{}&||\\
\noalign{\smallskip}
V(\hat \Q)& \overset{g_{j,0}^{-1}}{\longrightarrow} & V(\hat\Q)
\end{matrix}
\end{equation}
Modifying the isometry $\mu_j$ by an element of $\text{\rm O}(W)(\hat F)$, if necessary, we can
assume that $g_{j,0}\in \SO(V)(\hat\Q)$.  For any element $g_j\in G(\hat \Q)$ with image $g_{j,0}$ in
$\SO(V)(\hat\Q)$,
the finite adele points of the tori $T(j)$ and $T$ are related, as subgroups of $G(\hat\Q)$,  by
\begin{equation}\label{conjTjT}
T(j)(\hat\Q) = g_j T(\hat\Q) g_j^{-1},
\end{equation}
and hence
\begin{equation}\label{conjKjK}
K_{T(j)}^{g_j} = g_j K_T g_j^{-1}.
\end{equation}
These relations depend only on the image $g_{j,0}$ of $g_j$.

The reciprocity laws for the
action of $\Aut(\C)$ on special points of Shimura varieties
\cite{milne.shih.I}, \cite{milne.shih.II}, \cite{milne}, yields the following result.

\begin{lemma}\label{bymilne} Let the notation be as above and let $\tau \in
\Aut(\mathbb C/\mathbb Q)$.\hfill\break
(1) If $\tau =\sigma_j \circ \sigma_0^{-1}$ on $\sigma_0(E)$, then there is a preimage $g_j$
of $g_{j,0}$, unique up to an element of $\Q^\times$, such that
$$
\tau(Z(T, h_0)) = Z(T(j), h_0(j),g_j).
$$
(2) If $\tau =\rho$ is complex conjugation,
then
$$
\tau(Z(T, h_0)) = Z(T, h_0^-).
$$
Here $h_0^-$ is the map from $\mathbb S$ to $G_\mathbb R$ induced
by $\mathbb S \rightarrow \Gspin(W_{\sigma_0}), z \mapsto \bar z$.
\end{lemma}

We will write
$$Z(T(j), h_0^\pm(j),g_j)= Z(T(j), h_0^+(j),g_j) + Z(T(j), h_0^-(j),g_j).$$
We will also write $z_0^\pm(j)\in \mathbb D$ for the oriented
negative two planes in $V(\mathbb R)$ associated  $h_0^\pm(j)$. Let
\begin{equation} \label{def:CM}
Z(W) = \sum_{j=0}^d Z(T(j), z_0^\pm(j),g_j) \in Z^{2d}(X_K)
\end{equation}
Then $Z(W)$ is a $0$-cycle defined over $\mathbb Q$.

\section{\bf Special divisors and automorphic  Green functions}
\label{sect3}

In this section, we briefly review the special divisors defined in
\cite{KuDuke} and their `automorphic' Green functions
defined by the first author and Funke using regularized theta liftings
\cite{Br}, \cite{BF}. We prove that these special cycles do not intersect with
the special cycles defined in Section \ref{sect2}.

 Let $x\in V(\Q)$ be a vector of positive norm. We write
$V_x$ for the orthogonal complement of $x$ in $V$ and $G_x$ for
the stabilizer of $x$ in $G$. So $G_x\cong \GSpin(V_x)$. The
sub-Grassmannian
\begin{align}
\label{eq:dx} \D_{x}=\{ z\in \D;\;\text{$z\perp x$}\}
\end{align}
defines an analytic divisor of $\D$. For $g\in G(\widehat \Q)$ we
consider the natural map
\begin{align} \label{eqY4.3}
G_x(\Q)\bs \D_x\times G_x(\widehat \Q)/(G_x(\widehat \Q)\cap gKg^{-1})
\longrightarrow X_K,\quad (z,g_1)\mapsto (z,g_1 g).
\end{align}
Its image defines a divisor $Z(x,g)$ on $X_K$, which is rational
over $\Q$. For $m\in \Q_{>0}$ and $\ph \in S(V(\widehat{\mathbb Q}))^K$, if there is an $x_0 \in V(\mathbb Q)$ with $Q(x_0)=m$, we
define the weighted cycle
\begin{equation}
Z(m, \ph)= \sum_{g \in G_{x_0}(\widehat\Q)\backslash G(\widehat{\mathbb Q})/K}  \ph(g^{-1}x_0) Z(x_0,g).
\end{equation}
It is a divisor on $X_K$ with complex coefficients. Note that,  since $\ph$ has compact support in $V(\widehat\Q)$ and the orbits of
$K$ on the compact set $G(\widehat \Q)\cdot x_0\cap \supp(\ph)$ are open,
the sum is finite. If there is no $x_0\in
V(\mathbb Q)$ such that $Q(x_0)= m$, we set $Z(m, \ph)=0$.

\begin{proposition} Let the notation be as above. Then $Z(m,\ph)$ and $Z(T(j), h_0^\pm(j),g_j)$ do not intersect in $X_K$.
\end{proposition}
\begin{proof} It suffices to show that $Z(x, g_1) \cap Z(T,
h_0,g_2)$ is empty for every $x \in V(\mathbb Q)$ with $Q(x)>0$ and $g_1, g_2
\in G(\widehat{\mathbb Q})$. Suppose $P=[z, hg_1]=[z_0, tg_2]$ is in the
intersection, where $z_0=\mathbb R e_0 +\mathbb R f_0$ is the
negative two plane associated to $h_0$, and $z$ is a negative
two-plane in $V(\mathbb R)$ which is orthogonal to $x$.  Then
there are $\gamma \in G(\mathbb Q)$ and $k \in K$ such that
$$
(\gamma)_\infty z = z_0,  \quad \hat{\gamma}h g_1 k =tg_2.
$$
Here $\hat{\gamma}$ is the image of $\gamma$ in $G(\widehat{\mathbb Q})$.
Let $y = \gamma x \in V(\mathbb Q)$. Then $x\perp z$ implies that $y
\perp z_0$, i.e., $(\sigma_0(y), e_0)=(\sigma_0(y), f_0)=0$. This
implies that $\sigma_0(y)=0$ and thus $y=0$, a contradiction.
\end{proof}

Let $L$ be an even integral lattice in $V$, i.e., $Q(x)=\frac12(x,x) \in
\mathbb Z$ for $x \in L$, and let
$$
L'=\{ y \in V:\, (x, y) \in \mathbb Z, \hbox{ for } x \in L \}
\supset L
$$
be its dual. For $\mu \in L'/L$, we write $\ph_\mu =\cha(\mu + \hat
L) \in S(V(\widehat{\mathbb Q}))$ and $Z(m, \mu) = Z(m, \ph_\mu)$,
where $\hat L =L \otimes \widehat{\mathbb Z}$. Associated to the
reductive dual pair $(\SL_2,\Orth(V))$ there is a  Weil representation
$\omega=\omega_{\psi}$ of $\SL_2(\A)$ on the Schwartz space
$S(V(\A))$, where $\psi$ is the `canonical' unramified additive
character of $\mathbb Q \backslash \A$ with $\psi_\infty(x) =e(x)$.
Since the subspace $S_L=\oplus\, \mathbb C \ph_\mu \subset
S(V(\widehat\Q))$ is preserved by the action of $\SL_2(\widehat\Z)$,
there is a representation $\rho_L$ of $\Gamma=\SL_2(\mathbb Z)$ on
this space defined by the formula
$$
\rho_L(\gamma) \ph =\bar{\omega} (\hat\gamma) \ph
$$
where $\hat\gamma \in \SL_2(\hat{\mathbb Z})$ is the image of
$\gamma$. This representation is given explicitly by Borcherds as
\begin{align}
\label{eq:weilt}
\rho_L(T)(\ph_\mu)&=e(Q(\mu^2))\,\ph_\mu,\\
\label{eq:weils}
\rho_L(S)(\ph_\mu)&=\frac{e((2-n)/8)}{\sqrt{|L'/L|}} \sum_{\nu\in
L'/L} e(-(\mu,\nu))\, \ph_\nu,
\end{align}
where
$T = \begin{pmatrix} 1&1\\0&1\end{pmatrix}$ and $S= \begin{pmatrix}0&1\\-1&0\end{pmatrix}$,
see e.g. \cite{Bo1}, \cite{KuIntegral}, \cite{Br}.  Note that the complex conjugate $\bar\rho_L$ is thus the restriction of
$\omega$ to the subgroup $\SL_2(\Z)\subset \SL_2(\widehat\Z)$.


 Recall that a smooth function  $f:\H\to S_L$ is called a {\em harmonic weak Maass form}
(of weight $k$ with respect to $\Gamma$ and ${\rho}_L$) if it
satisfies:
\begin{enumerate}
\item[(i)] $f \mid_{k,{\rho}_L} \gamma= f$
for all $\gamma\in \Gamma$; i.e.,
$$
f(\gamma \tau) =(c \tau + d)^k \rho_L(\gamma)f(\tau).
$$

\item[(ii)] there is a $S_L$-valued Fourier polynomial
\[
P_f(\tau)=\sum_{\mu\in L'/L}\sum_{n\leq 0} c^+(n,\mu)\, q^n\,\ph_\mu
\]
such that $f(\tau)-P_f(\tau)=O(e^{-\eps v})$ as $v\to \infty$ for
some $\eps>0$;\item[(iii)] $\Delta_k f=0$, where
\begin{align*}
\Delta_k := -v^2\left( \frac{\partial^2}{\partial u^2}+
\frac{\partial^2}{\partial v^2}\right) + ikv\left(
\frac{\partial}{\partial u}+i \frac{\partial}{\partial v}\right)
\end{align*}
is the usual weight $k$ hyperbolic Laplace operator (see
\cite{BF}).
\end{enumerate}
The Fourier polynomial $P_f$  is called the {\em principal part}
of $f$. We denote the vector space of these harmonic weak Maass
forms by  $H_{k,\rho_L}$. Any weakly holomorphic modular form is a
harmonic weak Maass form. The Fourier expansion of any $f\in
H_{k,\rho_L}$ gives a unique decomposition $f=f^++f^-$, where
\begin{subequations}
\label{deff}
\begin{align}
\label{deff+}
f^+(\tau)&= \sum_{\mu\in L'/L}\sum_{\substack{n\in \Q\\ n\gg-\infty}} c^+(n,\mu)\, q^n\,\ph_\mu,\\
\label{deff-} f^-(\tau)&= \sum_{\mu\in L'/L} \sum_{\substack{n\in \Q\\
n< 0}} c^-(n,\mu)\,\Gamma(1-k, 4 \pi |n| v)\, q^n\, \ph_\mu,
\end{align}
\end{subequations}
and,  for $a >0$, $\Gamma(s, a)= \int_{a}^\infty e^{-t}t^{s-1}\, dt$  is the
incomplete $\Gamma$-function. We refer to $f^+$ as the
{\em holomorphic part} and to $f^-$ as the {\em non-holomorphic
part} of $f$.

Recall that there is an antilinear  differential operator $\xi=
\xi_k:H_{k,\rho_L}\to S_{2-k,\bar\rho_L}$, defined by
\begin{equation}
\label{defxi} f(\tau)\mapsto \xi(f)(\tau):=2i v^{k}
\overline{\frac{\partial}{\partial\bar \tau} f(\tau)}.
\end{equation}
 By  \cite[Corollary~3.8]{BF}, one has  the exact
sequence
\begin{gather}
\label{ex-sequ}
\xymatrix{ 0\ar[r]& M^!_{k,\rho_L} \ar[r]& H_{k,\rho_L}
\ar[r]^{\xi}& S_{2-k,\bar\rho_L} \ar[r] & 0 }.
\end{gather}

Let $f\in H_{1-d,\bar\rho_L}$ be a harmonic weak Maass form of
weight $1-d$ with representation $\bar\rho_L$ for $\Gamma$, and
denote its Fourier expansion as in \eqref{deff}. Let $S_L^\vee$ be
the dual space of $S_L$---the space of linear functionals on $S_L$,
and let $\{ \ph_\mu^\vee \} $ be the dual  basis in $S_L^\vee$ of
the basis $\{\ph_\mu \}$ of $S_L$.  Recall that   the Siegel theta
function
$$
\theta_L(\tau, z, g) =\sum_{\mu} \theta(\tau, z, g,   \ph_\mu)\,
\ph_\mu^\vee
$$
is an $S_L^\vee$-valued holomorphic modular form of weight $d-1$ for
$\Gamma$ and $\rho_L $ defined as follows (see \cite[Section 2]{BY2}
or \cite{KuIntegral} for details). For $z \in \mathbb D$, one has
decomposition
$$
V(\mathbb R) = z \oplus z^\perp, \quad x =x_{z} + x_{z^\perp}.
$$
Let $(x, x)_z= -(x_z, x_z) + (x_{z^\perp}, x_{z^{\perp}})$ and
define the associated Gaussian by
\begin{equation}\label{gaussian}
\ph_\infty(x, z) = e^{-\pi (x, x)_z}.
\end{equation}
Then, for $\tau \in \mathbb H$, $[z, g]\in X_K$, and $\ph \in
S(V(\widehat{\mathbb Q}))$,  the theta function is given by
$$
\theta(\tau, z, g, \ph) = v^{\frac{1}2(1-d)}\sum_{x \in V(\mathbb
Q)} \omega (g'_\tau) \ph_\infty(x, z) \ph(g^{-1}x), \qquad g'_\tau =
\begin{pmatrix}
v^{\frac12}&uv^{-\frac12}\\{}&v^{-\frac12}\end{pmatrix}\in
\SL_2(\R).
$$
Here $g$ acts on $V$ via its image in $\SO(V)$.

 We consider the regularized theta integral
\begin{align}
\label{reg1} \Phi(z,g,f)=\int_{\calF}^{reg} \langle f(\tau),
\theta_L(\tau,z,g)\rangle\,d\mu(\tau)
\end{align}
for $z\in \D$ and $g\in G(\widehat\Q)$, where $\mathcal F$ is the
standard domain for $\SL_2(\mathbb Z) \backslash \mathbb H$. The
integral is regularized as in \cite{Bo1}, \cite{BF}, that is,
$\Phi(z,g,f)$ is defined as the constant term in the Laurent
expansion at $s=0$ of the function
\begin{align}
\label{reg2} \lim_{T\to \infty}\int_{\calF_T} \langle f(\tau),
\theta_L(\tau,z,g)\rangle\,v^{-s} d\mu(\tau).
\end{align}
Here $\calF_T=\{\tau\in \H; \; \text{$|u|\leq 1/2$, $|\tau|\geq 1$,
and $v\leq T$}\}$ denotes the truncated fundamental domain and the
integrand \begin{equation} \langle f(\tau),\theta_L(\tau,z,g)\rangle
=\sum_{\mu \in L'/L} f_\mu(\tau) \theta(\tau, z, g, \ph_\mu)
\end{equation}
is the pairing of $f$ with the Siegel theta function, viewed as a
linear functional on the space $S_L$.

The following theorem summarizes some properties of the function
$\Phi(z,g,f)$ in the setup of the present paper (see \cite{Br},
\cite{BF}).

\begin{theorem} \label{theoY4.2}
The function $\Phi(z,g,f)$ is smooth on $X_K\bs Z(f)$, where
\begin{align}
\label{eq:zf}
Z(f)=\sum_{\mu\in L'/L}\sum_{m>0} c^+ (-m,\mu) Z(m,\mu).
\end{align}
It has a logarithmic singularity along the divisor $-2Z(f)$. The
$(1,1)$-form $dd^c \Phi(z,g,f)$ can be continued to a smooth form
on all of $X_K$. We have the Green current equation
\begin{align} \label{eqY4.10}
dd^c[\Phi(z,g,f)]+\delta_{Z(f)}=[dd^c \Phi(z,g,f)],
\end{align}
where $\delta_Z$ denotes the Dirac current of a divisor $Z$.
Moreover, if $\Delta_z$ denotes the invariant Laplace operator on
$\D$, normalized as in \cite{Br}, we have
\begin{align}
\Delta_z  \Phi(z,g,f) = \frac{n}{4}\cdot c^+(0,0).
\end{align}
\end{theorem}

In particular, the theorem implies that $\Phi(z,g,f)$ a Green
function for the divisor $Z(f)$ in the sense of Arakelov geometry in
the normalization of \cite{SABK}. (If the constant term $c^+(0,0)$
of $f$ does not vanish, one actually has to work with the
generalization of Arakelov geometry given in \cite{BKK}.) Moreover,
we see that $\Phi(z,g,f)$ is harmonic when $c^+(0,0)=0$. Therefore,
it is called the {\em automorphic Green function} associated with
$Z(f)$. Notice also that $ Z(f)$ has coefficients in $\mathbb
Q(f)$, the field generated by the $c(-m, \mu)$, $m >0$.

\section{\bf CM values of Siegel theta  functions and Eisenstein series}
\label{sect4}

Recall that, for each $j$,  we have fixed an isomorphism $V\simeq V(j)=\RR_{F/\mathbb Q} W(j)$
of rational quadratic spaces, and hence an identification
\begin{equation}\label{identify}
S(V(\A_\Q)) = S(W(j)(\A_F)),\qquad \ph \mapsto \ph_{F,j}
\end{equation}
of the corresponding Schwartz spaces. For example, if $\ph_F=\ttt_w\ph_{F,w}\in S(W(j)(\A_F))$, with $w$ running over the places of
$F$, then the corresponding $\ph\in S(V(\A_\Q))$ is also factorizable, with local component
$\ph_v = \ttt_{w\mid v} \ph_{F,w}$ in the space
$$S(\RR_{F/\Q}W(j)(\Q_v)) = S(\oplus_{w\mid v} W(j)(F_w)) = \ttt_{w\mid v} S( W(j)(F_w)).$$
These identifications are compatible with the Weil representations of
$\SL_2(\A_\Q)$ and $\SL_2(\A_F)$ for our fixed additive character $\psi$ of $\A_\Q$ and the
character $\psi_F = \psi\circ\tr_{F/\Q}$ of $\A_F$, i.e.,
$$\omega_{V,\psi}(g')\ph = \omega_{W(j), \psi_F}(g')\ph_{F,j},$$
where, on the right side, we view $g'\in \SL_2(\A_\Q)$ as an element of $\SL_2(\A_F)$. We write
$\ph_F$ for $\ph_{F,0}$.  Moreover, we will frequently
abuse notation and write $\ph$ for $\ph_F$ and identify $S(W(\A_F))$ with $S(V(\A))$.  Note that the Weil representations $\omega_{W(j),\psi_F}$ of $\SL_2(\A_F)$,
which are now all realized on $S(V(\A_\Q))$, via (\ref{identify}), do not coincide in general. The point is that the group
$\SL_2(F)$ in the dual pair
$(\SL_2(F),\RR_{F/\Q}\text{\rm O}(W(j)))$ arises as the commutant in the ambient symplectic group of the subgroup
$\RR_{F/\Q}\text{\rm O}(W(j)) \subset \text{\rm O}(V)$,
i.e., by a seesaw construction, and these subgroups do not coincide.

Recall that, for each $j$, we have fixed an isometry $\mu_j: W(j)(\hat F) \overset{\sim}{\longrightarrow} W(\hat F)$,
and an element $g_{j,0}\in \SO(V)(\hat\Q)$ so that the diagram (\ref{seesaw.relations}) commutes.
\begin{lemma}\label{twisting.Weil}
(i) For any $\ph\in S(V(\hat\Q))$, recall that we identify $\ph_{F, 0}=\ph_F$ with $\phi$ via $S(W(\hat F)) \cong S(V(\hat\Q))$.  Then
$$\mu_j^*(\ph) = (\omega(g_{j,0})\ph)_{F,j}.$$
(ii) The map  $\mu_j^*: S(W(\hat F)) \rightarrow S(W(j)(\hat F))$
intertwines the Weil representations  $\omega_{W,\psi_F}$ and $\omega_{W(j),\psi_F}$ of $\SL_2(\hat\Q)$
on these spaces.\hfill\break
(iii) For $g'\in \SL_2(\hat\Q)$, and $\ph\in S(V(\hat\Q))$,
$$\omega_{W(j),\psi_F}(g')\omega(g_{j,0})\ph = \omega(g_{j,0})\,\omega_{W,\psi_F}(g')\ph.$$
\end{lemma}
Here in part (iii), we are working in the fixed space $S(V(\hat\Q))$ with natural linear action of $g\in\SO(V)(\hat\Q)$, $\omega(g)\ph(x) = \ph(g^{-1}x)$,
and the various Weil representation actions of $\SL_2(\hat F)$, as described above.

For $z\in \mathbb D$, the Gaussian $\ph_\infty(\cdot ,z)\in S(V(\R))$ is defined by (\ref{gaussian}).
The points $z_0^\pm(j)\in \mathbb D$ are the fixed points of $T(j)(\R)$, and
$$\ph_\infty(\cdot ,z_0^\pm(j)) = \ttt_i \ph_{\infty,W(j)_{\sigma_i}},$$
in
$$S(V(\R)) = S(\RR_{F/\Q}(W(j))(\R)) = \ttt_i S( W(j)_{\sigma_i}),$$
where $W(j)_{\sigma_i}=W(j)\otimes_{F,\sigma_i}\mathbb R$, and
$$
\ph_{\infty, W(j)_{\sigma_i}} (x) = e^{-\pi |(x,x)_{W(j)_{\s_i}}|}
$$
is the Gaussian for the definite space $W(j)_{\sigma_i}$.
%
Note that $\ph_{\infty, W(j)_{\sigma_i}}$ is $\SO(W_{j,\sigma_i})$-invariant, and is an eigenfunction of $\SO_2(\mathbb
R) \subset \SL_2(\mathbb R)$ with respect to the Weil
representation $\omega_{W_{j, \sigma_i}}$ of `weight' $+ 1$ for $i\ne j$ and $-1$ for $i=j$.

For a $K$-invariant Schwartz function $\ph \in S(V(\hat{\mathbb Q}))^K$ and $\tau \in \mathbb H$,  the theta function
\begin{equation}\label{thetaoverQ}
\theta(\tau, z, g, \ph) =v^{\frac{1-d}2} \sum_{x \in V(\mathbb Q)}
\omega_V(g'_\tau)\ph_\infty(x, z) \ph(g^{-1}x)
\end{equation}
is an automorphic function of $[z,g]\in X_K$, where $z\in \mathbb D$ and $g\in G(\widehat \Q)$.
By the preceding discussion, the pullback of this function to $Z(T(j),z_0^\pm(j),g_j)$
coincides with the pullback of the Hilbert theta function associated to the quadratic space $W(j)$,
\begin{equation}\label{thetaoverF}
\theta(\vec\tau, t, (\omega(g_{j,0})\ph)_{F,j}) =
v_j\,N(\vec v)^{-\frac12}\sum_{x\in W(j)(F)} \omega_{W(j)}(g'_{\vec\tau})\ph_{\infty,W(j)}(x)\,(\omega(g_{j,0})\ph)_{F,j}(t^{-1}x),
\end{equation}
via the diagonal embedding of $\mathbb H$ into $\mathbb H^{d+1}$. Here $\vec\tau\in \mathbb H^{d+1}$, with components $\tau_r = u_r+iv_r$,
$N(\vec v) = \prod_r v_r$, and $g'_{\vec\tau}\in \SL_2(\R)^{d+1}$
with component $g'_{\tau_r}$ in the $r$th slot. This theta function has weight
$$\bold 1(j):=(1,\dots, -1,\dots, 1),$$
with $-1$ in the $j$th slot.

  Let $\chi=\chi_{E/F}$ be the quadratic Hecke
  character of $F$ associated to $E/F$, and let $I(s, \chi)=\otimes_v
  I(s, \chi_v)$ be the representation of $\SL_2(\A_F)$ induced from the character $\chi\ |\ |^s$
  of the standard Borel subgroup. We
  write $\Phi_{\sigma_i}^k$ for  the unique eigenfunction of
  $\SO_2(\mathbb R) \subset \SL_2(F\otimes_{F, \sigma_i} \mathbb
  R)$ in $I(s, \chi_{\sigma_i})$ of weight  $k$ with
  $\Phi_{\sigma_i}^k(1, s)=1$.  We define sections in $I_\infty(s,\chi_\infty) = \ttt_i I(s,\chi_{\sigma_i})$ by
  $$\Phi_\infty^{\bold 1}(s) = \ttt_i \Phi_{\sigma_i}^1(s),$$
  and
  $$\Phi_\infty^{\bold 1(j)}(s) = \Phi_{\sigma_j}^{-1}(s)\ttt(\ttt_{i\ne j}\Phi_{\sigma_i}^1(s)).$$
 For each $j$, there is an $\SL_2(\widehat{F})$-equivariant map
  $$
  \lambda_j: S(W(j)(\widehat{F})) \rightarrow I_f(0, \chi_f), \qquad\ph
  \mapsto \lambda_j(\ph) (g) =\omega_{W(j),\psi_F}(g)\ph(0).
  $$
  By (ii) of Lemma~\ref{twisting.Weil},  these maps for various $j$'s  are related as follows.
  \begin{lemma} For $\ph \in S(V(\hat{\Q}))$, one has
  $$\lambda_j(\mu_j^*(\ph_F)) = \lambda_0(\ph_F)$$
  \end{lemma}
  Let $\Phi_{\ph}(s) \in I_f(s, \chi_f)$ be the unique
  standard section with $\Phi_{\ph}(g, 0) =\lambda_0(\ph) = \lambda_j(\mu_j^*(\ph))$.  For
  $\ph \in S(W(\widehat F))=S(V(\widehat\Q))$ and $\vec\tau=(\tau_0,\dots, \tau_d)\in \mathbb H^{d+1}$
  with $\tau_r = u_r+iv_r$,, we define the
  Hilbert-Eisenstein series
  \begin{equation}
  E(\vec\tau, s, \ph,\bold 1) =N(\vec v)^{-\frac12}
  E ( g'_{\vec\tau}, s,  \Phi_\infty^{\bold 1}\ttt\Phi_{\ph})
\end{equation}
and
\begin{equation}
E(\vec\tau, s, \ph,\bold 1(j)) = v_j\,N(\vec v)^{-\frac{1}2}
  E ( g'_{\vec\tau}, s, \Phi_\infty^{\bold 1(j)}\ttt \Phi_{\ph}).
\end{equation}
Here $N(\vec v) = \prod_r v_r$.
Note that,
$\Phi_\infty^{\bold 1(j)}(s)$ is associated to the Gaussian $\ph_{\infty,W(j)}$, so that
$E(\vec\tau, s, \ph,\bold 1(j))$ is a coherent Eisenstein series of weight $\bold 1(j)$ attached to
the function $\ph_{\infty,W(j)}\otimes \mu_j^*(\ph)\in S(W(j)(\A_F))$
and
$E(\vec\tau, s, \ph,\bold 1)$ is an incoherent Eisenstein series of parallel weight $\bold 1$ (independent of $j$).
The two Eisenstein series are related as follows by an observation
of \cite[(2.17)]{KuIntegral}, \cite[Lemma 2.3]{BY2},

\begin{lemma}  \label{lem4.1} Write $\bar{\partial}_{j} =
\frac{\partial\phantom{\bar\tau} }{\partial\bar\tau_{\sigma_j}} d
\bar{\tau}_{\sigma_j}$. Then
$$
-2 \bar{\partial}_{j} \left( E'(\vec\tau,0,\ph,\bold 1)\,d\tau_{\sigma_j}\right) = E(\vec\tau, 0, \ph,\bold 1(j))
\,d\mu(\tau_{\sigma_j}).
$$
\end{lemma}

In this paper, we normalize the Haar measure $dh$ on
$\SO(W(j))(\A_F)$ so that
$$\vol(\SO(W(j))(F)\backslash \SO(W(j))(\A_F))=2,$$
and write $dh = dh_\infty\,dh_f$ where
$dh_\infty = \prod_i dh_{\infty_i}$ with
$\vol(\SO(W(j)_{\sigma_i}),dh_{\infty_i}) =1$.
For the
convenience of the reader,
we first recall \cite[Lemma 2.13]{Scho}.
\begin{lemma} \label{lem4.2}  For any
function $f$ on
$$Z(T(j), z_0(j),g_j)= T(j)(\Q) \backslash \big(\,\{z_0(j) \}
\times T(j)(\hat\Q)/K^{g_j}_{T(j)}\,\big),$$
the weighted sum (\ref{weighted.sum}) of the values of $f$ over this discrete finite set
is given by
$$
f(Z(T(j), z_0(j),g_j)) =\frac12\,\deg Z(T, z_0) \int_{\SO(W(j))(F)
\backslash \SO(W(j))(\hat{F})} f(z_0(j), t) \,  dt.
$$
Here
$$
\deg Z(T, z_0) =\frac{4}{\vol(K_T)}
$$
is independent of $j$.
\end{lemma}
\begin{proof}
By \cite[Lemma 2.13]{Scho}, the formula holds with
$\deg Z(T, z_0)$ replaced by the quantity $2/\vol (K^{g_j}_{T(j)}) $. So it
suffices to check $ \vol(K^{g_j}_{T(j)})=\vol(K_T)$  is independent of
$j$. But this is immediate by (\ref{conjTjT}) and (\ref{conjKjK}). \end{proof}

\begin{proposition} \label{prop4.2} With the notation as above,
$$
 \theta(\tau, Z(T(j),z_0(j),g_j),
\ph) =
C\cdot E(\tau\DD, 0, \ph,\bold 1(j))
$$
where
$$C= \frac12\, \deg(Z(T,z_0)).$$
\end{proposition}

\begin{proof} Since $\vol(\SO(W(j))_{\sigma_i})=1$, one has by Lemma \ref{lem4.2} that
$$
\theta(\tau, Z(T(j),z_0(j),g_j) = \frac12\,\deg Z(T, z_0)
\int_{\SO(W(j))(F) \backslash \SO(W(j))(\A_F)} \theta(\tau\DD, t, (\omega(g_{j,0})\ph)_{F,j})\, dt,
$$
where the theta function in the integral is given by (\ref{thetaoverF}).
Now the proposition follows from the Siegel-Weil formula,
\end{proof}

For $\chi=\chi_{E/F}$ as above,  let
\begin{equation} \label{eq:L-series}
\Lambda(s, \chi) =A^{\frac{s}2} (\pi^{-\frac{s+1}2}
\Gamma(\frac{s+1}2))^{d+1} L(s, \chi),  \quad A=N_{F/\mathbb Q}
(\partial_F d_{E/F})
\end{equation}
 be the complete $L$-function of $\chi$. It is a holomorphic function of $s$ with
functional equation
$$
\Lambda(s, \chi) =\Lambda(1-s, \chi),
$$
and 
$$
\Lambda(1, \chi) =\Lambda(0, \chi) =L(0, \chi)=2^{d-\delta}\frac{h(E)}{w(E)\,h(F)} \in \mathbb Q^\times,
$$
where $2^\delta = |\calO_E^\times:\mu(E) \calO_F^\times|$ is $1$ or $2$.
Let
$$
E^*(\vec\tau, s, \ph,\bold 1) = \Lambda(s+1, \chi)\, E(\vec\tau, s, \ph,\bold 1)
$$
be the normalized incoherent Eisenstein series.

\begin{proposition}\label{prop4.3} Let $\ph=\ph_F \in S(V(\hat\Q)) =S(W(\hat F))$.
For  a totally positive element $t\in F^\times_+$,
let $a(t,\ph)$ be the $t$-th Fourier coefficient of $E^{*,
\prime}(\vec\tau, 0, \ph,\bold 1)$ and write
the constant term  of  $E^{*, \prime}(\vec\tau, 0, \ph,\bold 1)$ as
$$\ph(0)\, \big(\,\Lambda(0, \chi)  \log \norm(\vec v) +a_0(\ph)\,\big).$$
Let
$$
\mathcal E(\tau, \ph) = \ph(0)\,  a_0(\ph) + \sum_{n \in \mathbb
Q_{>0}} a_n(\ph)\, q^n
$$
where
$$
a_n(\ph) =\sum_{t \in F^\times_+, \, \tr_{F/\mathbb Q} t =n} a(t,\ph).
$$
Then, writing $\tau\DD$ for the diagonal image of $\tau\in \mathbb H$ in $\mathbb H^{d+1}$,
$$
E^{*, \prime}(\tau\DD, 0, \ph,\bold 1) -\mathcal E(\tau, \ph)-\ph(0)\, \Lambda(0, \chi)\,
(d+1)\,\log v
$$
is of exponentially decay as  $v$ goes to infinity.
Moreover, for $n >0$
$$
a_n(\ph)  =\sum_p a_{n, p}(\ph) \log p
$$
with $a_{n, p}(\ph) \in \mathbb Q(\ph)$, the subfield of $\mathbb C$
generated by the values $\ph(x)$, $x\in V(\widehat\Q)$.
\end{proposition}
\begin{proof}
Let $\mathcal C =\otimes_v \mathcal C_v$ be the incoherent
collection of local quadratic $F_v$-spaces with $\hat{\mathcal C}=\hat{W}$
for finite adeles and  $\mathcal C_{\infty} $ is totally positive definite.
Then
$$\Phi_{\ph}(0)\ttt \Phi_\infty^{\bold 1}(0)=\lambda(\ph\otimes
\ph_{\infty,\mathcal C})$$
for $\ph \otimes \ph_{\infty ,\mathcal C}\in S(\mathcal C) = S(\hat{W})\ttt S(\mathcal C_\infty)$,
where $\ph_{\infty,\mathcal C} = \ttt_i \ph_{\infty, \mathcal C_{\sigma_i}}$ is the product of the Gaussians
for the positive definite binary quadratic spaces $\mathcal C_{\sigma_i}$.
Thus $E(\vec\tau, s, \ph,\bold 1)$ is an incoherent Eisenstein series according
to \cite{KuAnnals} and $E^*(\vec\tau, 0, \ph,\bold 1)=0$. By linearity, we may assume that the function
$\ph= \ttt_v \ph_v\in S(W(\widehat F))$ is factorizable,
the Fourier expansion can be written as
$$
E^*(\vec\tau, s, \ph,\bold 1) =E_0^*(\vec\tau, s, \ph,\bold 1) +\sum_{t \in F^\times}
E_t^*(\vec\tau, s, \ph,\bold 1)
$$
with
$$
E_t^*(\vec\tau, s, \ph,\bold 1) =A^{\frac{s}2} \prod_{\mathfrak p < \infty} W_{t,
\mathfrak p}^*(1, s, \ph_\mathfrak p) \prod_{i=0}^d W_{\sigma_i(t),
\sigma_i}^*(\tau_i, s, \Phi_{\sigma_i}^1)
$$
and
$$
E_0^*(\vec\tau, s, \ph,\bold 1) = \ph(0) \Lambda(s+1, \chi) \norm(\vec v)^{\frac{s}2}
+A^{\frac{s}2} \prod_{\mathfrak p < \infty} W_{0, \mathfrak p}^*(1, s, \ph_\mathfrak p)
\prod_{i=0}^d W_{0, \sigma_i}^*(\tau_i, s,
\Phi_{\sigma_i}^1).
$$
Here, for $g'\in \SL_2(F_\mathfrak p)$,
$$
W_{t, \mathfrak p}^*(g', s, \ph_\mathfrak p) =L_\mathfrak p(s+1,\chi_v)\, W_{t, \mathfrak p}(g', s, \ph_\mathfrak p)
$$
and
$$
W_{\sigma_i(t), \sigma_i}^*(\tau_i, s, \Phi_{\sigma_i}^1)
 =\pi^{-\frac{s+2}2} \Gamma(\frac{s+2}2)\,v_i^{-\frac{1}2}\, W_{\sigma_i(t),
 \sigma_i}(g'_{\tau_i}, s, \Phi_{\sigma_i}^1)
 $$
 are the normalized local Whittaker functions, which are computed in
 \cite{KRY1} and \cite{YaValue} in special cases. In particular,
 \cite[Proposition 2.6]{KRY1} (see also \cite[Proposition 1.4]{YaValue}) asserts that\footnote{
The extra `$-$' in the formula is due to the fact that we use
$w=\kzxz {0} {-1} {1} {0}$ here for the  local Whittaker function
instead of $w^{-1}$ in \cite{KRY1}.
}
\begin{align*}
W_{\sigma_i(t), \sigma_i}^*(\tau_i, 0,
\Phi_{\sigma_i}^1)&=2 \,\gamma(\mathcal C_{\sigma_i})\, e(\sigma_i(t)\tau_i),&&\ff
\sigma_i(t)
>0,
\\
 W_{0, \sigma_i}^*(\tau_i, s,
\Phi_{\sigma_i}^1)&=\gamma(\mathcal C_{\sigma_i})\, v_i^{-s/2}
\pi^{-\frac{s+1}2} \Gamma(\frac{s+1}2),
\\
W_{\sigma_i(t), \sigma_i}^*(\tau_i, 0,
\Phi_{\sigma_i}^1)&=0, &&\ff \sigma_i(t) <0.
\end{align*}
Here
$\gamma(\mathcal C_{\sigma_i})$ is the local Weil index, an $8$-root
of unity.  Moreover, in the last case,
$$
W_{\sigma_i(t), \sigma_i}^{*, \prime}(\tau_i, 0,
\Phi_{\sigma_i}^1)=\gamma(\mathcal C_{\sigma_i})\,
e(\sigma_i(t)\tau_i)\, \beta_1(4 \pi |\sigma_i(t)| v_i)
$$
is of exponentially decay when $v_i$ goes to infinity. Here
$$
\beta_1(x) =\int_1^\infty e^{-xt} t^{-1} \, dt,  \quad x >0.
$$
On the other hand, when everything is unramified at a finite prime
$\mathfrak p$,   i.e., $E/\mathbb Q$ is unramified at primes over $\mathfrak p$, $\alpha \in
\OO_{F_{\mathfrak p}}^\times$, and $\ph_{\mathfrak p} =\cha(\OO_{E_{\mathfrak p}})$, one has (\cite[Proposition
1.1]{YaValue}) for $t \ne 0$
$$
\gamma(\mathcal C_{\mathfrak p})^{-1}\, W_{t, v}^*(1, s, \ph_{\mathfrak p})  =\begin{cases}
   0 &\ff t \notin \OO_{\mathfrak p},
   \\
    \ord_{\mathfrak p} t +1 &\ff \mathfrak p \hbox{ split in } E/F, t \in \OO_{\mathfrak p},
    \\
     \frac12(1+ (-1)^{\ord_{\mathfrak p}t}) &\ff \mathfrak p \hbox{ inert in } E/F, t \in
     \OO_{\mathfrak p}.
     \end{cases}
     $$
In general, $\gamma(\mathcal C_{\mathfrak p})^{-1} W_{t, v}^*(1, s, \ph_{\mathfrak p})$ is a
polynomial of $\norm(\mathfrak p)^{-s}$ with coefficients in $\mathbb Q(\ph_{\mathfrak p})$ (\cite{KYEisenstein}). For $t \ne
0$, let $ D(t)=D(t,\mathcal C) $ be  the `Diff' set of places $\mathfrak p$ of
$F$ (including infinite places) such that $\mathcal C_{\mathfrak p}$ does not represent $t$,  as
defined in \cite{KuAnnals}. Then $D(t)$ is a finite set of odd
order, and for every $\mathfrak p \in D(t)$, the local Whittaker function at
$v$ vanishes at $s=0$. So $E_t^{*, \prime}(\vec\tau, 0, \ph)=0$
unless $D(t)$ has exactly one element. Assuming this and
restricting $\vec\tau$ to the diagonal $\tau\DD=(\tau, \cdots, \tau)$ with $\tau =u+ \sqrt{-1} v \in \mathbb H$, there are two subcases.

When $D(t)=\{\sigma_i\}$ for some $i$, the above formulae shows that
 $$
 E_t^{*, \prime}(\tau\DD, 0, \ph,\bold 1) = W_{\sigma_i(t), \sigma_i}^{*,
 \prime}(\tau, 0, \Phi_{\sigma_i}^1) \prod_{\mathfrak p\ne \sigma_i} W_{t,
 \mathfrak p}^*(\cdot, 0, \cdot)
 $$
 is of exponential decay when $v=\text{\rm Im}(\tau) \rightarrow \infty$.

 When $D(t)=\{\mathfrak p \}$ for some finite prime $\mathfrak p$,   $t \in F^\times_+$ is totally positive,
 $$
 E_t^{*, \prime}(\tau\DD, 0, \ph,\bold 1)=a(t, \ph)  \, q^{\tr_{F/\mathbb
 Q}t}, \quad  q=e(\tau)
 $$
 for some $a_t(\ph) \in \mathbb Q(\ph) \log p$, where $p$ is the prime
 below $\mathfrak p$. Here we have used the fact that
$$
\prod_{\mathfrak p <\infty } \gamma(\mathcal C_{\mathfrak p})\prod_{i=1}^{d+1} \gamma(\mathcal C_{\sigma_i}) =-1.
$$
 Finally, for the constant term, one has (see
 e.g., \cite[Section 1]{YaValue} or \cite{KYEisenstein})
 $$
 E_0^{*} (\vec\tau, s, \ph,\bold 1)
  = \ph(0) \left(\norm(\vec v)^{\frac{s}2}\,\Lambda(s+1, \chi)  + \norm(\vec v)^{-\frac{s}2} \Lambda(1-s,
  \chi) M_\ph(s)\right)
  $$
where $M_\ph(s)$ is a product of finitely many polynomials in
$\norm(\mathfrak p)^{-s}$ for finitely many `bad' $\mathfrak p$,  and $M_\ph(0)=-1$.
Recalling that $E_0^*(\tau\DD, 0, \ph,\bold 1) =0$, this gives for $\tau \in \mathbb H$
\begin{equation}
E_{0}^{*, \prime} (\tau\DD, 0, \ph,\bold 1)= \ph(0)\, \left(\Lambda(1, \chi)
(d+1)\log v + 2 \Lambda'(1, \chi) + \Lambda(1, \chi) M_\ph'(0)\right).
\end{equation}
The constant
term  of $E^{*, \prime}(\tau\DD, 0, \ph,\bold 1)$ as a (non-holomorphic)
 elliptic modular form is
$$
E_0^{*, \prime}(\tau\DD, 0, \ph,\bold 1) + \sum_{0\ne t \in F,
\tr_{F/\mathbb Q} t =0} E_{t}^{*, \prime}(\tau\DD, 0, \ph,\bold 1),
$$
where the last sum is of exponential decay when $v=\text{\rm Im}(\tau)
\rightarrow \infty$.  This proves the proposition.
\end{proof}

\section{\bf The main formula} \label{sect5}

Let $L$ be an even integral lattice in $V$, and let $K \subset
G(\hat{\mathbb Q})$ be a compact open subgroup which fixes $L$ and
acts trivially on $L'/L$. We also assume that $K$ satisfies the condition
\begin{equation}\label{fullness}
K\cap \G_m(\widehat\Q) = \widehat \Z^\times,
\end{equation} where $\G_m$ is the kernel of the homomorphism
$\GSpin(V)\rightarrow \SO(V)$. Let $f \in H_{1-d, \bar\rho_L}$ be a
harmonic weak Maass form and let $\Phi(z, h, f)$ be the
corresponding `automorphic' Green function for the divisor $Z(f)$ defined in (\ref{eq:zf}).

For $\vec\tau \in \mathbb H^{d+1}$ and $\tau \in \mathbb H$, define $S_L^\vee$-valued functions by
 \begin{equation}
 E(\vec\tau, s, L,\bold 1) =\sum_{\mu \in L'/L} E(\vec\tau, s, \ph_\mu, \bold 1)\,
 \ph_\mu^\vee,  \qquad  \mathcal E(\tau, L) =\sum_{\mu \in L'/L} \mathcal
 E(\tau, \ph_\mu)\, \ph_\mu^\vee,
\end{equation}
 where $\mathcal E(\tau,\ph)$ is defined in Proposition~\ref{prop4.3}, and  the normalized incoherent Eisenstein series
 $$
 E^*(\vec\tau, s,L, \bold 1) =\Lambda(s+1, \chi)\, E(\vec\tau, s, L,\bold 1).
 $$
Define the $L$-function for an cuspidal modular form $g=\sum_\mu
g_\mu \ph_\mu \in S_{d+1, \rho_L}$
\begin{equation}
\LL(s, g, L) = \langle E^*(\tau\DD, s, L,\bold 1), g\rangle_{\Pet} :=
\int_{\SL_2(\mathbb Z) \backslash \mathbb H}  \sum_\mu
\overline{g_\mu(\tau)}\, E^*(\tau\DD, s, \ph_\mu,\bold 1)\, v^{d+1}\,  d\mu(\tau).
\end{equation}
It can be viewed as the $g$-isotypical component of diagonal
restriction of  the Hilbert-Eisenstein series $E(\vec\tau, s,
L,\bold 1)$.

\begin{remark} \label{rem5.1}  The Eisenstein series $E(\vec\tau, s, L, \bold 1)$ depends on the
$F$-quadratic form on $L \otimes \mathbb Q =W$, not just on the $\Q$-quadratic form  on  $L \otimes \mathbb Q=V$.
When we need to emphasize this dependence on the
$F$-quadratic form,  we will write $L(W)$ rather than  $L$   and
\begin{align*}
E^*(\vec\tau, s, L(W), \bold 1) =E^*(\vec\tau, s, L, \bold 1), \quad \mathcal E(\tau, L(W))= \mathcal E(\tau, L), \quad \LL(s, g, L(W))=\LL(s, g, L)
\end{align*}
We also caution that $L(W)$ might not be an $\OO_F$-lattice, i.e., it might not be $\OO_F$-invariant.
\end{remark}

Since the Eisenstein series has an analytic continuation and
is incoherent, the $L$-series  $\LL(s, g, L)$ has an analytic
continuation and is zero at the central point $s=0$. Now we are ready
to state and prove the main formula.  Here, if $\sum_na_n q^n$ is a
power series in $q$, we write
$$\CT\big[\,\sum_{n} a_n q^n\,\big] = a_0$$
for the constant term.

\begin{theorem} \label{theo5.1} For a harmonic weak Maass form $f\in H_{1-d,\bar\rho_L}$
with components $f= f^++f^-$ as in (\ref{deff}) and with other notation as above,
$$
\Phi(Z(W), f) = C(W,  K) \left(\ \CT\big[\,\langle f^+(\tau), \mathcal
E(\tau, L(W))\rangle\,\big] + \LL'(0, \xi(f),L(W))\ \right).
$$
where $\xi(f)$ is the image of $f$ under the anti-holomorphic operator
$\xi: H_{1-d,\bar\rho_L}\to S_{d+1,\rho_L}$, cf. (\ref{defxi}),
and
$$
C(W, K)=\frac{ \deg(Z(T,z_0^\pm))}{  \Lambda(0, \chi)}.
$$
\end{theorem}
\begin{proof} The proof basically follows the same argument as  in \cite[Theorem
4.8]{BY2}. We write $L$ in place of $L(W)$.  First,  by Lemma \ref{lem4.1} and  Proposition
\ref{prop4.2}, we have
\begin{align*}
\Phi(Z(T(j), z_0(j),g_j),f) &=
 \int_{\mathcal F}^{\reg} \langle f(\tau),
\theta_L(\tau, Z(T(j), z_0(j),g_j))\rangle\,  d\mu(\tau)
 \\
 &= C \int_{\mathcal F}^{\reg} \langle f(\tau), E(\tau\DD, 0, L, \bold 1(j))\,d\mu(\tau)\rangle
 \\
  &= -2 C \int_{\mathcal F}^{\reg} \langle f(\tau), \bar{\partial}_{j}( E'(\tau\DD, 0,
  L,\bold 1)\,d\tau)\rangle.
  \end{align*}
Here $C$ is the constant in Proposition \ref{prop4.2}. So, summing on $j$,
and recalling the definition (\ref{def:CM}) of $Z(W)$, we have
\begin{align*}
\Phi(Z(W), f) &=-4\,C \int_{\mathcal F}^{\reg} \langle f(\tau),
\sum_j\bar{\partial}_{j}( E'(\tau\DD, 0,
  L, \bold 1)\, d\tau)\rangle
  \\
   &= -4 C \int_{\mathcal F}^{\reg} \langle f(\tau),
\bar{\partial}( E'(\tau\DD, 0,
  L, \bold 1)\,d\tau)\rangle
  \\
   &= -4 C \int_{\mathcal F}^{\reg} d( \langle f(\tau),
   E'(\tau\DD, 0, L,\bold 1)\,d\tau\rangle)  +4 C  \int_{\mathcal F}^{\reg} \langle \bar\partial f(\tau),
   E'(\tau\DD, 0, L, \bold 1)\,d\tau\rangle
   \\
   &=-C_0 I_1 +4 C_0 I_2,
   \end{align*}
   where $C_0=4 C \Lambda(0, \chi)^{-1}=C(W, K)$, and
   \begin{align*}
   I_1 &=\int_{\mathcal F}^{\reg} d( \langle f(\tau),
   E^{*, \prime}(\tau\DD, 0, L,\bold 1)\,d\tau\rangle),
   \\
   I_2&=\int_{\mathcal F}^{\reg} \langle \bar\partial f(\tau),
   E^{*, \prime}(\tau\DD, 0, L,\bold 1)\,d\tau\rangle.
   \end{align*}
 Recall that
 $$
 \bar\partial f(\tau) =- \frac{1}{2i} v^{d-1}
 \overline{\xi(f)}\, d \bar\tau.
 $$
Thus
 $$
\langle \bar\partial f(\tau),
   E^{*, \prime}(\tau\DD, 0, L,\bold 1)\,d\tau\rangle =-\langle \overline{\xi(f)},
   E^{*, \prime}(\tau\DD, 0, L,\bold 1)\rangle\, v^{d+1}\, d\mu(\tau)
   $$
   is actually integrable over the fundamental domain $\mathcal F$, and hence
   $$
   I_2=-\int_{\mathcal F}\langle \overline{\xi(f)},
   E^{*, \prime}(\tau\DD, 0, L,\bold 1)\rangle\, v^{d+1} d\mu(\tau)
   =-\LL'(0, \xi(f), L).
   $$
 By the  same argument as in \cite[Proposition 2.5]{KuIntegral},
 \cite[Proposition 2.19]{Scho}, or \cite[Lemma 4.6]{BY2}, there is a
 (unique) constant $A_0$ such that
 $$
 I_1 =\lim_{T \rightarrow  \infty} \left(\int_{\mathcal F_T} d( \langle f(\tau),
   E^{*, \prime}(\tau\DD, 0, L,\bold 1)\,d\tau\rangle) -A_0 \log T\right)=\lim_{T \rightarrow
   \infty} (I_1(T)-A_0 \log T).
$$
By Stokes'  theorem, one has
\begin{align*}
 I_1(T) &=\int_{\partial \mathcal F_T}  \langle f(\tau),
   E^{*, \prime}(\tau\DD, 0, L,\bold 1)\rangle\, d\tau
   \\
    &= -\int_{iT}^{iT+1} \langle f(\tau),
   E^{*, \prime}(\tau\DD, 0, L,\bold 1)\rangle\, du
   \\
    &=-\int_{iT}^{iT+1} \langle f^+(\tau),
   E^{*, \prime}(\tau\DD, 0, L,\bold 1)\rangle\, du +O(e^{-\epsilon T})
   \end{align*}
   for some $\epsilon >0$ since $f^-$ is of exponential decay and
   $E^{*, \prime}$ is of moderate growth.
Proposition~\ref{prop4.3} asserts that
$$
E^{*, \prime}(\tau\DD, 0, L) =\mathcal E(\tau, L) + \Lambda(0,
\chi)\, (d+1)\,\log(v) + \sum_{\mu \in L'/L} \sum_{m \in \mathbb Q} a(m,
\mu, v) q^m
$$
such that $a(m, \mu, v) q^m$ is of exponentially decay as $v
\rightarrow \infty$.  Thus,
$$
-I_1(T)= \CT[\langle f^+(\tau),\mathcal E(\tau, L)\rangle] +
\Lambda(0, \chi)\,(d+1) \log T +\sum_{\mu \in L'/L} \sum_{m+n=0} c^+(m,
\mu) a(n , \mu, T).
$$
The last sum goes to zero when $T\rightarrow \infty$.  So we can
take $A_0=(d+1)\,\Lambda(0, \chi)$, and
$$
I_1=-\CT[\langle f^+(\tau),\mathcal E(\tau, L)\rangle]
$$
as claimed.
\end{proof}

\begin{remark} There is a sign error in front of $\LL'(\xi(f), U, 0)$ in \cite[Theorem 4.7]{BY2} and throughout that
paper caused by this error. The $+\LL'(\xi(f), U, 0)$ in that theorem should be $-\LL'(\xi(f), U, 0)$.
Accidently, in the proof of \cite[Theorem 7.7]{BY2},
there is another sign error relating the Faltings' height and the Neron-Tate height.
Two wrong signs give the correct formula in \cite[Theorem 7.7]{BY2},
which somehow prevented the authors from discovering the sign error earlier.
\end{remark}

As in \cite{BY2}, this theorem raises two interesting conjectures.
We very briefly describe them and refer to \cite[Section 5]{BY2} for
details.  Assume that there is a regular scheme $\calX_K\to
\Spec\Z$, projective and flat over $\Z$, whose associated complex
variety is a smooth compactification $X_K^c$ of $X_K$. Let $\calZ(m
,\mu)$ and $\calZ(W)$ be suitable extensions to $\calX_K$ of the
cycles $Z(m, \mu)$ and $Z(W)$, respectively. Such extensions can be
found in low dimensional cases  using a moduli interpretation of
$\calX_K$.
%
For an $f \in H_{1-g, \bar{\rho}_L}$, the function $\Phi(\cdot,f)$
is a Green function for the divisor $Z(f)$. Set $\calZ(f)
=\sum_\mu \sum_{m>0} c^+(-m, \mu) \calZ(m, \mu)$. Then
the pair
\[
\hat \calZ(f)=(\calZ(f),\Phi(\cdot,f))
\]
defines an arithmetic divisor in $\widehat\CH^1(\calX_K)_\C$.
  Theorem
\ref{theo5.1} provides a formula for the quantity
\begin{align} \label{eq:4.23} \langle \hat \calZ(f), \calZ(W)
\rangle_{\infty}=\frac{1}{2}\Phi(Z(W),f),
\end{align}
and inspires the following  `equivalent' conjectures.

\begin{conjecture}
\label{conj5.2} Let $\mu\in L'/L$,  and let $m\in Q(\mu)+\Z$ be
positive. Then $\calZ(m ,\mu)$ and $\calZ(W)$ intersect properly,
and  the arithmetic intersection number $\langle \calZ(m, \mu),
\calZ(W)\rangle_{fin}$ is equal to $-\frac{1}2 C(W, K)$ times the $(m,\mu)$-th
Fourier coefficient of $ \calE(\tau, L)$.
\end{conjecture}

\begin{conjecture} \label{conj5.3}
For any $f\in H_{1-d, \bar{\rho}_L}$,
one has
\begin{align}
\label{eq:fh} \langle \hat \calZ(f), \calZ(W) \rangle_{Fal} =\frac{1}2 \,C(W, K)
\left( c^+(0,0)\kappa(0,0)-\LL'( 0, \xi(f), L)\right).
\end{align}
Here $\kappa(0, 0)$ is the constant term of $\mathcal E(\tau, L)$
\end{conjecture}

\section{\bf Hilbert modular surfaces} \label{Hilbert}

In this section, we study the case of Hilbert modular surfaces.  Let $F= \mathbb Q(\sqrt D)$ be a real quadratic field of discriminant $D$
with non-trivial Galois automorphism $\sigma$ and different $\partial= \partial_F$.
Let
\begin{equation} \label{eq6.1}
V = \{ M = \kzxz {u} {b \sqrt D} {a/\sqrt D} {\sigma(u)}:\,  a, b \in \Q,  u \in F\}\simeq F\oplus \Q^2,
\end{equation}
with quadratic form
$$
Q(M) =\det(M) = N_{F/\Q}(u) - ab,
$$
of signature $(2,2)$, and let $L$ be the lattice $\OO_F\oplus \Z^2$.
Let $G$ be the algebraic group over $\Q$ such that, for any
$\Q$-algebra $R$,
\begin{equation}\label{GL2F0}
G(R) =\{ g \in \GL_2(F\otimes_\Q R):\,  \det g \in R^\times \}.
\end{equation}
Then  $G \cong \Gspin(V)$ and the action of $G$ on $V$ is given by
$$
g\cdot A = \sigma(g) A g^{-1}.
$$
Let
\begin{equation}
K =\{ \abcd \in G(\hat{\Q}):\,  a, d \in \hat{\OO}_F, c \in
\hat{\partial}^{-1}, b \in \hat{\partial}\}.
\end{equation}
Note that the dual lattice of $L$ is $L' \simeq \partial^{-1}\oplus \Z^2$.
Then it is easy to check that $K$ preserves $L$ and acts trivially on $L'/L$.
By the strong approximation theorem, one has $G(\hat{\Q})=G(\Q)_+ K$ and so
$$X_{K}(\C) = G(\Q)\backslash \big(\,\mathbb D\times G(\hat\Q)/K\,\big) \simeq X=\Gamma\backslash \mathbb H^2,$$
where
$$
\Gamma =\{ g =\abcd \in \text{\rm SL}_2(F):\, a, d \in \OO_F,  c \in \partial^{-1},  b \in \partial \}
$$
and
$$\mathbb D \simeq (\mathbb H^{\pm}\times \mathbb H^\pm)_0.$$
Here the subscript indicates the set of pairs $(z_1,z_2)\in \mathbb H^{\pm}\times \mathbb H^\pm$ such that $\text{\rm Im}(z_1)\text{\rm Im}(z_2)>0$.
Thus, $X_K(\C)$ is a Hilbert modular surface. 

In fact, the canonical model of the Shimura variety $X_K$ over $\Q$ is the coarse moduli scheme over $\Q$ of isomorphic classes of principally polarized abelian surfaces
with real multiplication,  $\bold A = (A, \kappa_0, \lambda)$,
$\kappa_0: \OO_F\rightarrow \End(A)$, \cite[Section 1.27]{rapoport.HB}.
Hence $X = X_K(\C)$ can be naturally identified with the set of isomorphism classes of such objects over $\C$.
This interpretation allows us to define CM $0$-cycles as follows.

\subsection{CM $0$-cycles} Let $E$ be a non-biquadratic quartic CM field with real quadratic subfield $F = \Q(\sqrt{D})$ with fundamental discriminant $D$, and let $\Sigma=\{\sigma_1, \sigma_2\}$ be a fixed   a CM type of $E$. Let $\tilde E$ be the reflex field of $(E, \Sigma)$, the subfield of $\C$ generated by the type norms
\begin{equation} \label{eq:Typenorm}
\norm_\Sigma(z) = \sigma_1(z) \sigma_2(z), \quad z \in E.
\end{equation}
Then $\tilde E$ is also a  quartic CM  number field with real subfield $\tilde F =\Q(\sqrt{\tilde D})$ if
 the absolute discriminant of $E$ is  $d_E =D^2 \tilde D$. Notice that, in general, $\tilde D$  is not the fundamental discriminant of $\tilde F$.

Let $\CM^\Sigma(E)$ be the set of isomorphic classes of principally polarized CM abelian surfaces
 $\bold A=(A, \kappa, \lambda)$  over $\C$ of CM-type $(\OO_E, \Sigma)$: $A$ is a CM abelian surface over $\mathbb C$ with an $\OO_E$-action
  $\kappa: \OO_E \hookrightarrow \End(A)$  and a principal polarization $\lambda: A \rightarrow A^\vee$ satisfying the further conditions:
 (i) the Rosati involution induced by $\lambda$ induces the complex conjugation on
 $E$, and  (ii)  there are two translation invariants, non-zero differentials $\omega_1$ and $\omega_2$ on $A$ over $\C$ such that
$$
\kappa(r)^* \omega_i= \sigma_i(r)\, \omega_i,  \quad r \in \OO_E,  i=1, 2.
$$
There is a natural map (of sets of isomorphism classes)
$$j^\Sigma:\CM^\Sigma(E) \lra X, \qquad \bold A= (A,\kappa,\lambda) \mapsto (A,\kappa|_F,\lambda).$$
We also let
$$j:\CM(E) = \coprod_{\Sigma} \CM^\Sigma(E) \lra X,$$
so that $\CM(E)$ defines a $0$-cycle on $X$.
 The main purpose  of this section is to use Theorem \ref{theo5.1} to derive a formula for $\Phi(\CM(E), f)$ for
 any $f \in H_{0, \bar{\rho}_L}$, which is a generalization of \cite[Theorem 1.4]{BY1}.

\subsection{$\CM(E)$ as an orbit space I} \label{sect6.1}

Given $\bold A=(A, \kappa, \lambda) \in \CM^\Sigma(E)$, let  $M=H_1(A, \Z)$ with the induced $\OO_E$-action and the non-degenerate symplectic form
$$
\lambda:  M \times M  \rightarrow \Z
$$
coming from the polarization of $A$.  In particular, $\lambda$ defines a perfect pairing on $M$ and satisfies
$$
\lambda(\kappa(r) x, y) =\lambda(x, \kappa(\bar r)y), \quad r \in  \OO_E,  x, y \in  M,
$$
so that $(M, \kappa, \lambda)$ is an $\OO_F$-polarized
CM module in the sense of \cite{HY}.  The action $\kappa$ makes  $M$
 a projective $\OO_E$-module of rank one, isomorphic to a
fractional ideal $\mathfrak A\subset E$. The  polarization $\lambda$ induces
a polarization $\lambda_\xi$ on $\mathfrak A$ given by
\begin{equation}\label{xi-polar}
\lambda_\xi: \, \mathfrak A \times \mathfrak A \rightarrow \Z, \quad
\lambda_\xi(x, y) = \tr_{E/\Q} \xi \bar x  y.
\end{equation}
where $\xi \in E^\times$ with
$\bar \xi =-\xi$.
A simple calculation shows that $\lambda$ is a principal polarization if and only if
\begin{equation} \label{eq6.5}
\mathfrak a:=\xi \partial_{E/F} \mathfrak A \bar{\mathfrak A} \cap F = \partial^{-1}.
\end{equation}
Moreover, $\bold A$ is of CM type $\Sigma$ if and only if $\Sigma(\xi) =(\sigma_1(\xi), \sigma_2(\xi)) \in \mathbb H^2$,
see for example \cite[Lemma 3.1]{BY1}.

 The converse is also true;   given $(\mathfrak A, \xi)$ satisfying (\ref{eq6.5}), there is
 a unique CM type $\Sigma$ of $E$ such that $\Sigma(\xi) \in \mathbb H^2$, and  one has
$$
\bold A(\mathfrak A, \xi):=(A=(\mathfrak A\otimes 1)
\backslash (E\otimes_\Q \R), \kappa, \lambda_\xi) \in \CM^\Sigma(E).
$$
Here we identify $E\otimes \R$ with $\C^2$ via the CM type $\Sigma$.

Two such  pairs $(\mathfrak A, \xi_\mathfrak A)$ and $(\mathfrak B,
\xi_\mathfrak B)$ are equivalent if there is an $r \in E^\times$ such that
$$
\mathfrak B =r \mathfrak A, \quad  r \bar r \xi_\mathfrak B =
\xi_\mathfrak A.
$$
Let $\PF(E)$ be the set of
equivalence classes of pairs $(\mathfrak A, \xi_\mathfrak A)$
satisfying (\ref{eq6.5}), and let $\PF^\Sigma(E)$ be the subset of
$(\mathfrak A, \xi_\mathfrak A)$ with $\Sigma(\xi_\mathfrak A) \in
\mathbb H^2$. Let $C(E)=I(E)/P(E) $ be the generalized ideal class
group of $E$, where $I(E)$ is the group of pairs $(\mathfrak Z,
\zeta)$ where $\mathfrak Z$ is a fractional ideal of $E$ and $\zeta
\in F^\times$ with
$$
\mathfrak Z \bar{\mathfrak Z} = \zeta \OO_E,
$$
and $P(E)$ is the subgroup of pairs $(r\OO_E,  r \bar r)$ for $r\in E^\times$.
The group $C(E)$ acts on $\PF(E)$ by
\begin{equation}\label{CEaction}
(\mathfrak Z,\zeta)\action (\mathfrak A, \xi) = (\mathfrak Z\mathfrak A,\zeta^{-1}\xi).
\end{equation}
The
following lemma is easy to check and is left to the reader.

\begin{lemma} Let the notation be as above. Then\hfill\break
(1) The map $(\mathfrak A, \xi) \mapsto \bold A(\mathfrak A,
\xi)$ gives a bijection between $\PF(E)$ and $\CM(E)$ and between $\PF^\Sigma(E)$
and $\CM^\Sigma(E)$.\hfill\break
(2)  The action of the group $C(E)$ on $\PF(E)$ defined by (\ref{CEaction}) is simply transitively.
\end{lemma}
The action of $C(E)$ on $\PF(E)$ gives thus a simply transitive
action of $C(E)$ on $\CM(E)$, which can also be described via the Serre
tensor product construction as in \cite{HY}.

\subsection{Special endomorphisms}
To a pair $(\mathfrak A, \xi)$ in $\PF^\Sigma(E)$, we can associate a lattice
\begin{equation}
L(\mathfrak A, \xi)= \{ j \in \End(\mathfrak A):\,  j\circ \kappa(a) =
\kappa(\sigma(a)) \circ j,  \quad j^* =j \}
\end{equation}
of special endomorphisms, and let $V(\mathfrak A, \xi) =L(\mathfrak A, \xi) \otimes \Q$.
Here $j^*$ is the adjoint of $j$ with respect to the pairing $\lambda_\xi$ defined by (\ref{xi-polar}).
If $\bold A= \bold A(\mathfrak A, \xi)$, we will also write $L(\bold A)$ and $V(\bold A)$ for
$L(\mathfrak A, \xi)$ and $V(\mathfrak A, \xi)$ respectively. In the notation  of \cite{HY}, $L(\bold A)$ is associated to   the  polarized CM module $M=H_1(\bold A)=\mathfrak A$ (not the  lattice of  special endomorphisms associated to the polarized CM abelian variety defined in \cite[Section 2]{HY}).
The following  slight refinement of \cite[Proposition 1.2.2]{HY}
shows that $Q_{\mathfrak A}(j)=j^2$ is a integral quadratic form on $L(\mathfrak A, \xi)$.

\begin{lemma}  \label{lem6.1} Let $(\mathfrak A, \xi) \in \PF^\Sigma(E)$.\hfill\break
(1) There are $\alpha, \beta \in \mathfrak A$ such that
$$
\mathfrak A = \OO_F \alpha + \partial^{-1} \beta, \quad    \xi\,
(\bar{\alpha} \beta - \alpha \bar{\beta}) =1.
$$
(2) Let $\alpha$ and $\beta$ be as in (1) and identify $\mathfrak A$ with $\OO_F \oplus
\partial^{-1}\, \subset F^2$ via
\begin{align*}
f:=f_{\alpha, \beta}: \,  &\mathfrak A \rightarrow  \OO_F \oplus
\partial^{-1}, \quad x \alpha + y \beta  \mapsto  \begin{pmatrix}x\\ y\end{pmatrix},
\\
\noalign{\noindent and define}
\kappa:=\kappa_{\alpha, \beta}:\,  &\OO_E \rightarrow
\End_{\OO_F}(\OO_F \oplus \partial^{-1}) \subset M_2(F),
\end{align*}
by
$$r (\alpha,\beta) \begin{pmatrix}x\\ y\end{pmatrix} = (\alpha,\beta)\,\kappa(r)\begin{pmatrix}x\\ y\end{pmatrix}.$$
Then the polarization $\lambda_\xi$ becomes the  standard symplectic
form $\lambda_{\st}$ on $F^2$ given by
$$
\lambda_{\st}( (x_1, y_1)^t, (x_2, y_2)^t) = \tr_{F/\mathbb Q} (x_1
y_2 -x_2 y_1).
$$
(3)
Moreover, define $j_0\in \End_{\Q}(F^2)$ by  $j_0 ((x, y)^t) = (\sigma(x), \sigma(y))^t$. Then,
for $V$ and $L$ given in (\ref{eq6.1}),  there is a $\Q$-linear isomorphism $V\isoarrow V(\mathfrak A, \xi)$
given by
$v\mapsto v \circ j_0$
and this isomorphism sends $L$ onto $L(\mathfrak A, \xi)$. \hfill\break
(4)   $Q_{\mathfrak A}(j)=j^2$ is a $\Q$-quadratic form on $V(\mathfrak A, \xi)$, and the quadratic lattice
 $(L(\mathfrak A, \xi), Q_{\mathfrak A, \xi}) \cong (L, Q)$ is independent of $(\mathfrak A, \xi)$.
\end{lemma}
\begin{proof}  It is easy to check using (\ref{eq6.5}) that the map
$$
\OO_F  \rightarrow \Hom_{\OO_F}(\Lambda_{\OO_F}^2 \mathfrak A,
\partial^{-1})=\Hom_{\mathbb Z}(\Lambda_{\OO_F}^2 \mathfrak A,
\mathbb Z), \quad a \mapsto \lambda_{a\xi}
$$
is an isomorphism (\cite[Lemma 3.1]{BY1}).

(1):  One can always write
$$
\mathfrak A = \OO_F \alpha + \mathfrak f \beta
$$
for some fractional ideal $\mathfrak f$ of $F$ and $\alpha, \beta
\in \mathfrak A$. Using the above isomorphism and explicit
calculation, one gets
$$\OO_F = (\xi (\bar{\alpha} \beta
-\alpha \bar{\beta}) \mathfrak f \partial_F)^{-1}.
$$
So replacing $\beta$ by $a\beta$ and $\mathfrak f $ by $a^{-1}
\mathfrak f$ if necessary, for some $a$, we may choose $\mathfrak f
=\partial^{-1}$, and $$ \xi (\bar{\alpha} \beta -\alpha \bar{\beta})
=1.
$$

For (2), a simple calculation gives
$$
\lambda_{\st}(f(z_1), f(z_2)) = \lambda_\xi(z_1, z_2)
$$
for $z_l=x_l \alpha + y_l \beta \in \mathfrak A$. Write $j = C\circ
j_0$. Then $j \circ \kappa(a) = \kappa(\sigma(a)) \circ j$ for all
$a \in F$ if and only if $C \in M_2(F)$. Write $w = \kzxz {0} {1}
{-1} {0}$. Then one has  for $z_l=(x_l, y_l)^t \in F^2$
\begin{align*}
\lambda_{\st}(j(z_1), z_2)&= \tr_{F/\Q} (\sigma(x_1), \sigma(y_1))
C^t w (x_2, y_2)^t
\\
 &= \tr_{F/\Q}( (x_1, y_1) \sigma(C)^t w (\sigma(x_2), \sigma(y_2))^t)
 \\
 &=  \tr_{F/\Q}( (x_1, y_1)w  \sigma(C)^\iota  (\sigma(x_2), \sigma(y_2))^t)
\end{align*}
So $j^* = \sigma(C)^\iota \circ j_0$. So $j \in V(\mathfrak A, \xi)$ if and
only if $C \in V$. Next, $ j \in  L(\mathfrak A, \xi) $ if and only if
$$
\kzxz {u} {b \sqrt D} {\frac{a}{\sqrt D}} {\sigma(u)} (\sigma(x),
\sigma(y))^t \in \OO_F \oplus \mathfrak a
$$
for all $(x, y)^t \in \OO_F \oplus \mathfrak a$, i.e., $C \in L$.

For (3),  one simply checks that  $j =C \circ j_0 \in L(\mathfrak A, \xi)$
satisfies  $j^2 =\det C = u \sigma(u) - ab$, which is an integral
quadratic form. So $(L(\mathfrak A, \xi), Q_{\mathfrak A}) \cong (L, Q)$ is independent
of $(\mathfrak A, \xi)$.
\end{proof}

\subsection{$\CM(E)$ as an orbit space II}


The following finer structure on $V(\mathfrak A, \xi)$ was discovered in \cite[Section 1]{HY}, see also \cite[Section 4]{BY1}.

\begin{proposition}\label{fromHY} Let $(\mathfrak A, \xi) \in \PF^\Sigma(E)$. Then
$V(\mathfrak A, \xi)$ has an $\tilde E$-vector space structure defined by
$$
\norm_\Sigma(r) \action j = \kappa(r) \circ j \circ \kappa(\bar r),
\quad  r \in E, \  j \in  V(\mathfrak A, \xi).
$$
Moreover there is a unique $\tilde F$-valued quadratic form $\tilde Q_{\mathfrak A}$ on
$V(\mathfrak A, \xi)$ such that
$$
Q_{\mathfrak A}(j) = \tr_{\tilde F/\Q} \tilde Q_{\mathfrak A}(j), \quad  \langle \tilde r \action
j_1, j_2\rangle_{\mathfrak A} = \langle j_1, \bar{\tilde r}\action j_2
\rangle_{\mathfrak A}
$$
for any $j \in V(\mathfrak A, \xi)$ and $\tilde r \in \tilde E$. Here $\langle\,
, \, \rangle_{\mathfrak A}$ is the symmetric bilinear $\tilde F$-form
associated to $\tilde Q_{\mathfrak A}$.
\end{proposition}
\begin{proof} This is described in detail in \cite[Section 1]{HY}.
We just replaced the CM field $E^\reflex$ in  \cite{HY} by  $\tilde E$, which
is isomorphic to $E^\reflex$ via  the reflex homomorphism  induced by
$a \otimes b \mapsto \sigma_1(a) \sigma_2(b)$.
\end{proof}

Let $W(\mathfrak A, \xi) =(V(\mathfrak A, \xi), \tilde Q_{\mathfrak A})$ be the $\tilde F$-quadratic space
associated to $(\mathfrak A, \xi)$, and let $\tilde L(\mathfrak A, \xi)=L(\mathfrak A, \xi)$ but with the $\tilde
F$-quadratic form $\tilde Q_{\mathfrak A}$. Recall that
$$
(\tilde L(\mathfrak A, \xi), \tr_{\tilde F/\Q} \tilde Q_{\mathfrak A}) \cong (L, Q)
$$
as quadratic $\Z$-lattices.   By Proposition~\ref{fromHY}, there is an $\alpha
\in \tilde E^\times$ such that  $W(\mathfrak A, \xi) \cong (\tilde E, \alpha z
\bar z)$,  and so $\SO(W(\mathfrak A, \xi)) =\tilde E^1$. Let $T_E$ be the algebraic
group over $\Q$ such that for any $\Q$-algebra $R$,
\begin{equation}
T_E(R) =\{ t \in (E \otimes R)^\times: \, t \bar t \in \mathbb
Q^\times \}.
\end{equation}
Note that the embedding $\kappa = \kappa_{\alpha,\beta}$ of (2) of Lemma~\ref{lem6.1} identifies $T_E$ with a maximal
torus in the group $G$ defined by (\ref{GL2F0}).
Let $S_{\tilde E}$ be the algebraic group over $\mathbb Q$ such that for
any $\Q$-algebra $R$,
\begin{equation}
S_{\tilde E}(R) =\{ t \in (\tilde E \otimes R)^\times: \, t \bar t =1  \}.
\end{equation}
In particular,  $S_{\tilde E}(\Q) = \tilde E^1$, and  $S_{\tilde E}
=\Res_{\tilde F/\Q}\SO(W(\mathfrak A, \xi))$.
Moreover, by \cite[Lemma 1.4.1]{HY}, the action of $T_E$ on $V(\mathfrak A, \xi)$ defined by
\begin{equation}
t\action j = \frac{1}{t \bar t}  \ \kappa(t) \circ j \circ \kappa(\bar t)
\end{equation}
determines an exact sequence
\begin{equation}
1\longrightarrow\G_m\longrightarrow T_E \overset{\nu_E}{\longrightarrow} S_{\tilde E}\lra 1, \quad \nu_E(t) =
\frac{t \otimes t}{t\bar t}.
\end{equation}
So, under the identification of $G$ with $\GSpin(V)$ given above,  $T_E$ is identified with
the maximal torus $T$ of $\GSpin(V)$ associated to $W(\mathfrak A, \xi)$ by the construction of Section 2.
Note the shift in notation (!), so that we are now writing $\tilde E$ for the field denoted by $E$ in Section 2.
By the construction of Section 2, we then have a CM cycle
$$
Z(W(\mathfrak A, \xi), z_0^\pm) = T(\Q) \backslash \{z_0^\pm \} \times
T(\hat{\Q})/U_E,
$$
where
$$
U_E=\{ r \in \hat{\OO}_E^\times: \, r \bar r \in \hat\Z^\times\} =K \cap
T(\hat{\Q}).
$$
Indeed, $U_E \subset K \cap T(\hat{\Q})$ and $U_E$ is a maximal
compact subgroup of $T(\hat{\Q})$, so $U_E=K \cap T(\hat{\Q})$.
 Let $C(T)
=T(\Q) \backslash T(\hat{\Q})/U_E$ be the `class group' of $T$. Define a homomorphism
$$
C(T) \rightarrow C(E),  \quad  [t] \mapsto [((t),  \zeta_t)],
$$
where $(t)=t\hat{\OO}_E \cap E$ is the ideal of $E$ associated to $t$, and $\zeta_t \in \Q_{>0}$ with
$$
\zeta_t \mathbb Z = t \bar t \hat{\mathbb Z} \cap \Q.
$$
Via this group homomorphism, $C(T)$ acts on $\CM(E)$.

Suppose that
$z_0^+=z_0^+(W(\mathfrak A, \xi))\in \mathbb D^+$
is such that the point $[z_0^+,1]\in X_K(\C)=X$
corresponds to the isomorphism class of $\bA=\bA(\mathfrak A, \xi) \in \CM^\Sigma(E)$.
Then, for $t\in T(\hat \Q)$,
$[z_0^+, t]$ corresponds to $t\action \bA$. The points $z_0^\pm$ go to the same
point in $X$ since $\diag(1, -1) \in G(\Q) \cap K$. On the other hand
$\bar\bA=(A, \bar\kappa, \lambda) \in  \CM^{\bar{\Sigma}}(E)$ also has the same image in $X$ as $\bA$.
So we can view $Z(W(\mathfrak A, \xi), z_0^+)$ as the $C(T)$-orbit of $\bA$ in $\CM^{\Sigma}(E)$
and $Z(W(\mathfrak A, \xi), z_0^-)$ as the $C(T)$-orbit of $\bar{\bA}$ in $\CM^{\bar{\Sigma}}(E)$.
Let $\Sigma'$ and $\bar{\Sigma}'$ be the other two CM types of $E$.

\begin{lemma} \label{lem6.4}  (1)  For any $t \in C(T)$ and $(\mathfrak A, \xi) \in \PF^\Sigma(E)$,
there is an isomorphism
$$( W(t\action (\mathfrak A, \xi)), \tilde Q_{t\action \mathfrak A}) \cong ( W(\mathfrak A, \xi), \tilde Q_{\mathfrak A})$$
of quadratic spaces over $\tilde F$ inducing
an isomorphism $( \hat{L}(t \action (\mathfrak A, \xi)), \tilde Q_{t\action \mathfrak A}) \cong ( \hat{L}(\mathfrak A, \xi), \tilde Q_{\mathfrak A})$.
\hfill\break
(2)  There is a class $\bold c=(\mathfrak Z, \zeta) \in C(E)$ such that $\bold c\action (\mathfrak A, \xi)$ has  CM type $\Sigma'$ for every
$(\mathfrak A, \xi) \in \PF^\Sigma(E)$ and
$(\hat{L}(\bold c \action (\mathfrak A, \xi)), \tilde Q_{\bold c\action \mathfrak A}) \cong (\hat{L}(\mathfrak A, \xi), \tilde Q_{\mathfrak A})$.
\end{lemma}
\begin{proof} (1) 
It is clear that
$$
\hat f:  \, \hat{L}(\mathfrak A, \xi)  \mapsto   \hat{L}(t \action (\mathfrak A, \xi))
\quad j \mapsto  j_t,
$$
 is a $\hat{\Z}$-linear isomorphism, where $j_t(tm ) = j(m)$ for $m \in \hat{\mathfrak A}$.  It is also clear that the induced isomorphism
$\hat f_\Q: \hat{W}(\mathfrak A, \xi)\rightarrow \hat{W}(t \action \mathfrak A, \xi)$ is
$\hat{\tilde E}$-linear.  One checks
$$
Q(j_t) = j_t^2 =j^2 =Q(j), \quad \hbox{ for all } j \in W(\mathfrak A, \xi).
$$
So, by uniqueness, one has
$$
\tilde Q_{t \action \mathfrak A}(j_t) = \tilde Q_{\mathfrak A}(j),   \quad \hbox{
for all } j \in \hat{W}(\mathfrak A, \xi).
$$
Thus,  $\hat f$ is an $\hat{\tilde F}$-quadratic isomorphism
which sends $\hat{L}(\mathfrak A, \xi)$ onto $\hat{L}(t\action (\mathfrak A, \xi))$. On the other hand, since
both $(\mathfrak A, \xi)$ and $t\action (\mathfrak A, \xi)$ are of the CM type $\Sigma$, there is an  $\tilde F_\infty$-quadratic
isomorphism $f_\infty$ between $W(\mathfrak A, \xi)_\infty$ and $W(t\action (\mathfrak A, \xi))_\infty$ by \cite[Proposition 1.3.5]{HY}.
By the Hasse principle, one has $W(\mathfrak A, \xi) \cong W(t\action (\mathfrak A, \xi))$.  This proves (1). Claim (2) is \cite[Proposition 1.4.3]{HY}.
\end{proof}

\begin{corollary}  \label{cor6.5} For $\bA =\bA(\mathfrak A, \xi)\in \CM^\Sigma(E)$,
let $Z(\bA)=Z(W(\mathfrak A, \xi))$ be the CM cycle defined in Section \ref{sect2}.
Then, as a subset of $\CM(E)$, $Z(\bA)$ is the union of the $C(T)$-orbits
of $\bA$, $\bar{\bA}$, $\bold c \action \bA$, and $\bold c \action \bar{\bA}$,
where $\bold c$ is a fixed element in $C(E)$ satisfying the condition in Lemma \ref{lem6.4}.
\end{corollary}

\begin{remark} \label{rem6.6}  By the definition in Section \ref{sect2}, each
point in $Z(\bA)$ is counted with multiplicity $\frac{2}{w_E}$, where $w_E$ is
the number of roots of unity in $E$. Furthermore, since $z_0^+$ and $z^-_0$ go to the same
point in $X$ (resp. $\bA$ and $\bar{\bA}$ go to the same point in $X$),  the image of
a point in $\CM(E)$ in $X$ is counted with multiplicity $\frac{4}{w_E}$ in this paper.
\end{remark}

\begin{remark}  A key point here is that the $0$-cycle $\CM(E)$ associated to the
non-biquadratic CM field $E/F$ via moduli coincides with a union of $0$-cycles associated to the
quadratic spaces $W(\mathfrak A, \xi)$ for the non-biquadratic CM fields $\tilde E/\tilde F$
via the Shimura variety construction of Section 2.  Note that the Shimura variety $\text{\rm Sh}(G,\mathbb D)$
for $G=\GSpin(V)$ is only PEL in this case due to an accidental isomorphism; this accounts for the
duality between the roles of the fields $E/F$ and $\tilde E/\tilde F$.
\end{remark}
Combining Theorem \ref{theo5.1} with Corollary \ref{cor6.5}, one has the following theorem.

\begin{theorem} \label{theo6.7}  Let  $f \in H_{0, \bar{\rho}_L}$, and let $E$ be a  non-biquadratic CM quartic field
with real subfield $F$ and a CM type $\Sigma$.  Let
$$
c(E) =\frac{4}{w_E} \frac{|C(T)|}{\Lambda(0, \chi)}.
$$
(1) For  $\bA \in \CM^\Sigma(E)$, we have
$$
\Phi(Z(\bA), f) = c(E) \left( \CT[\langle f^+, \mathcal E(\tau, \tilde L(\bA))\rangle] - \LL'(0, \xi(f), \tilde L(\bA))\right).
$$
(2) We have
$$
\Phi(\CM(E), f) =c(E) \sum_{\bA \in C(T) \backslash \CM^\Sigma(E)}
\left( \CT[\langle f^+, \mathcal E(\tau, \tilde L(\bA))\rangle] - \LL'(0, \xi(f), \tilde L(\bA))\right).
$$
\end{theorem}
Note that, for $\bA = \bA(\mathfrak A, \xi)$, we are writing $\tilde L(\bA)$ for $\tilde L(\mathfrak A, \xi)$
and $W(\bA)$ for $W(\mathfrak A, \xi)$. Notice also that $\tilde L(\bold A)$ is  $L=L(V)$ with  $\tilde F$-quadratic form $\tilde Q_{\bold A}$.
\begin{proof} By Theorem \ref{theo5.1},  and Corollary \ref{cor6.5}, one has
$$
\Phi(Z(\bA), f) = c_1(E) \left( \CT[\langle f^+, \mathcal E(\tau, \tilde L(\bA))\rangle] - \LL'(0, \xi(f), \tilde L(\bA))\right)
$$
with
$$
c_1(E) = \frac{\deg Z(W(\bold A), z_0^\pm)}{\Lambda(0, \tilde\chi)}.
$$
Here $\tilde\chi$ is the quadratic Hecke character of $\tilde E$ associated to $\tilde E/\tilde F$. By the proof of \cite[Proposition 3.3]{YaColmez}, one has
\begin{equation} \label{eq:Hyecke}
\Lambda(s, \chi) =\Lambda(s, \tilde\chi).
\end{equation}
So $c_1(E) = c(E)$ by Remark \ref{rem6.6}. Claim (2) follows from (1) and Corollary \ref{cor6.5}.
\end{proof}

\subsection{Integral structure}

In this section, we assume that $d_E=D^2 \tilde D$ with $D\equiv 1 \mod 4$
prime and $\tilde D \equiv 1 \mod 4$ square free, and give a more explicit formula for the CM value $\Phi(\CM(E), f)$.
Let $\Sigma$ be again a CM type of $E$ and let $\tilde E$ be its reflex field.
Consider the $\tilde F$-quadratic space
 \begin{equation}\label{deftW}
 \tilde W=\tilde E,  \quad \tilde Q(z) = - \frac{z \bar
 z}{\sqrt{\tilde D}}
 \end{equation}
 with even integral
 lattice $\tilde L = \OO_{\tilde E}$.
 \begin{proposition} \label{prop:onelattice} Let the notation  and assumption be as above. Then for any $\bold A \in \CM^\Sigma(E)$, there is an $\tilde F$-quadratic isomorphism
 $$
 \phi_{\bold A}:  (W(\bold A),  \tilde Q_{\bold A})  \cong  (\tilde W, \tilde Q)
 $$
 such that $\phi_{\bold A}(\tilde L(\bold A))$ is  in the same genus as $\tilde L$. In particular $\tilde L(\bold A)$ is an $\OO_{\tilde F}$-module and all $(\tilde L(\bold A) ,  \tilde Q_{\bold A})$ are in the same genus.
 \end{proposition}
\begin{proof}  Let  $\bA=\bA(\mathfrak A, \xi) \in \CM^\Sigma(E)$ and let $\alpha$ and  $\beta$ be chosen as in  Lemma \ref{lem6.1}. Define
$$
\phi:  W(\bA)=V  \rightarrow \tilde E,  \quad \phi(A) = \frac{1}{\sqrt D} (\sigma_1(\alpha), \sigma_1(\beta)) A
w (\sigma_2(\alpha), \sigma_2(\beta))^t, \quad w =\kzxz {0} {1} {-1} {0}.
$$
We first prove that  $\phi$ is an isomorphism of quadratic spaces over $\tilde F$ between
$(W(\bA), \tilde Q_\bA))$ and $(\tilde E, -\frac{z \bar z}{\sqrt{\tilde D} \norm(\mathfrak A)})$. To verify the claim,
let
\begin{equation} \label{spaceV1}
V_1 =\{ A \in M_2(F):\,  \sigma(A) = A^t\} =\{ \kzxz {b} {u}
{\sigma(u)} {a}:\, a, b \in \Q,  \, u \in F\}
\end{equation}
with quadratic form $Q_1(A) =D \det A$, and let
\begin{align}
\phi_1: \,  &V \rightarrow V_1,  \quad A \mapsto  \frac{1}{\sqrt D}
A w,
\\
\phi_2: \,  &V_1 \rightarrow E(\Sigma),  \quad A \mapsto
(\pi_1(\alpha), \pi_1(\beta)) A (\pi_2(\alpha),  \pi_2(\beta))^t.
\end{align}
Then $\phi =\phi_2 \circ \phi_1$. It is easy to check that $\phi_1$ is a $\Q$-isomorphism between
$(V, \det )$ and $(V_1, D \det )$. On the other hand, $\phi_2$ is  basically the map in \cite[(4.11)]{BY1},
and is a $\Q$-isomorphism  between $(V_1, D\det)$ and
$(\tilde E,  -\tr_{\tilde F/\Q} \frac{z \bar z}{\sqrt{\tilde D}\norm(\mathfrak A)})$. So $\phi$ is a $\Q$-quadratic
space isomorphism. Next,
For $\tilde r=\norm_\Sigma(r) \in  \tilde E$ with $r \in E^\times$, one has
\begin{align*}
\phi(\tilde r \action A)
 &=\frac{1}{\sqrt D} (\sigma_1(\alpha), \sigma_1(\beta)) \kappa(r) A
 \sigma(\kappa(r)^\iota) w (\sigma_2(\alpha), \sigma_2(\beta))^t
 \\
  &=\frac{1}{\sqrt D} (\sigma_1(r \alpha), \sigma_1(r \beta)) A w
  \sigma(\kappa(r)^\iota)^t (\sigma_2(\alpha), \sigma_2(\beta))^t
  \\
   &= \frac{1}{\sqrt D} (\sigma_1(r \alpha), \sigma_1(r \beta)) A w (\sigma_2(r \alpha), \sigma_2( r \beta))^t
   \\
    &= \sigma_1(r) \sigma_2(r) \phi(A)=\tilde r \phi(A).
\end{align*}
 So $\phi$ is $\tilde E$-linear. So $\phi$ is an $\tilde F$-quadratic space isomorphism
 between $(W(\bA), \tilde Q_\bA)$ and $(\tilde E,  -\frac{z \bar z}{\sqrt{\tilde D}\norm(\mathfrak A)})$, as claimed.

Second,  $$
L^0(\partial_F^{-1})=\{ \kzxz {b} {\lambda} {\sigma(\lambda)} {a} \in
V_1:\,  a \in \frac{1}D \mathbb Z,\,  b \in \mathbb Z,  \lambda \in
\partial^{-1} \}
$$
is a lattice in $(V_1, D \det )$. Then, by \cite[Proposition 4.7]{BY1},  one has
\begin{align*}
 \phi(\tilde L(\bA))
  &= \phi_2 \phi_1(L)
  \\
  &= \phi_2( L^0(\partial_F^{-1}))
  \\
   &= \norm_{\Sigma}(\mathfrak A).
\end{align*}
Here $\norm_{\Sigma}(\mathfrak A)$ is the type norm of $\mathfrak A$ defined as
$$
\norm_\Sigma(\mathfrak A) = \sigma_1(\mathfrak A) \sigma_2(\mathfrak A) \OO_M \cap \tilde E
$$
for any Galois extension $M$ of $\Q$ containing both $E$ and $\tilde E$.  Thus,  $\tilde L(\bA)$ is actually a fractional ideal in $\tilde E$, and in particular an $\OO_{\tilde F}$-lattice, and we have
\begin{equation} \label{neweq6.19}
\phi:  \tilde L(\bA)=(L(\bA), Q_\bA) \cong  (\norm_\Sigma(\mathfrak A), -\frac{z \bar z}{\sqrt{\tilde D}\norm(\mathfrak A)}).
\end{equation}

Third, we prove that for every  $\bA=\bA(\mathfrak A, \xi) \in \CM^\Sigma(E)$, one has for every finite prime $\mathfrak p$ of $\tilde F$
\begin{equation} \label{neweq6.20}
(\norm_\Sigma(\mathfrak A)_\mathfrak p, -\frac{1}{\sqrt{\tilde D}} \frac{z\bar z}{ \norm(\mathfrak A)} ) \cong (\tilde L_\mathfrak p,  \tilde Q).
\end{equation}
Notice that
(\cite[Corollary 4.5]{BY1})
\begin{equation}\label{6.18}
\norm_{\tilde E/\Q}  \partial_{\tilde E/\tilde F} = D,  \quad
\norm_{\tilde E/\tilde F} (\norm_{\Sigma}(\mathfrak A))=
\norm(\mathfrak A) \OO_F.
\end{equation}
In particular, $\tilde E/\tilde F$ is ramified at exactly one prime
$\mathfrak D$ of $\tilde F$ and this prime $\mathfrak D$ is above
$D$. For each prime ideal $\mathfrak p \ne \mathfrak D$ of $\tilde
F$, there is a generator $\alpha$ of $(\norm_\Sigma(\mathfrak
A))_\mathfrak p^\times$ such that $ \alpha \bar{\alpha}  =
\norm(\mathfrak A)$. So $r \mapsto r/\alpha$ gives
$$
((\norm_{\Sigma}\mathfrak A)_\mathfrak p,  -\frac{1}{\sqrt{\tilde D}} \frac{z \bar z}{ \norm(\mathfrak A)} )
\cong (\OO_{\tilde E, \mathfrak p},- \frac{1}{\sqrt{\tilde D}} z \bar z)=(\tilde L_\mathfrak p, \tilde Q).
$$
 For
$\mathfrak p =\mathfrak D$, one has similarly,
$$
(\norm_\Sigma(\mathfrak A), -\frac{z \bar z}{\sqrt{\tilde D}\norm(\mathfrak A)}) \cong (\OO_{\tilde E,
\mathfrak D}, -c \frac{1}{\sqrt{\tilde D}} z \bar z)
$$
for some some $c \in \OO_{\tilde F, \mathfrak D}^\times$ with
$$
\norm(\mathfrak A) = c  \alpha \bar \alpha, \quad \alpha \in \norm_\Sigma(\mathfrak A)_\mathfrak D.
$$
Since both $W(\bA)$ and $\tilde W$ are global $\tilde F$-vector
spaces, and  have thus global Hasse invariant $1$, one sees that
they have the same Hasse invariant at $\mathfrak D$ too. This
implies that
 $c \in \norm_{\tilde E_{\mathfrak D}/\tilde F_\mathfrak
D}(\OO_{\tilde E, \mathfrak D}^\times)$, and  one has again (\ref{neweq6.20}) for the prime $\mathfrak p=\mathfrak D$.
Finally, clearly $W(\bold A)$ and $\tilde W$ are    isomorphic at all infinite places of $\tilde F$ by  (\ref{neweq6.19}).
So there is an $\A_{\tilde F}$ isomorphism of $\A_{\tilde F}$-quadratic spaces
$$
\phi_{\bold A}': \quad  (W(\bold A)_\A,  \tilde Q_\bold A)  \cong  (\tilde W_\A, \tilde Q)
$$
such that $\phi_{\bold A}'(\hat{\tilde L}(\bold A)) = \hat{\tilde L}$. By the Hasse principle, one proves the proposition.
\end{proof}

 Let
 \begin{equation}
 E^{*}(\vec \tau, s, \tilde L, \bold 1) =\sum_{\mu \in \tilde
 L'/\tilde L} E^*(\vec\tau, s, \ph_\mu, \bold 1) \,\ph_\mu
 \end{equation}
be the associated incoherent Eisenstein series, and let $\mathcal
E(\tau, \tilde L)$ be the holomorphic part of $E^{*, \prime}(\tau^\Delta,  0, \tilde L)$ with $\tau \in \mathbb H$.
Note that
$$\tilde L'/\tilde L \simeq \partial^{-1}_{\tilde E/\tilde F}/\OO_{\tilde E} \simeq \Z/ D\Z.$$

\begin{remark} In \cite[Section 6]{BY1} (where our $\tilde F$ is denoted by $F$),  a slightly different  $\tilde F$-quadratic space  is used:
$$
\tilde W^+ = \tilde E,  \quad \tilde Q^+(z) =\frac{1}{\sqrt{\tilde D}} z \bar z
$$
with lattice $\tilde L^+$. This corresponds to $\CM^{\Sigma'}(E)$ where $\Sigma'$ is a CM type of $E$ which is not $\Sigma$ or its complex conjugation.
Notice that $-1 \in \norm_{\tilde E/\tilde F} \hat{\OO}_{\tilde E}^\times$. So $\hat{\tilde L}$ is isomorphic to $\hat{\tilde L}^+$, and  the associated incoherent Eisenstein series are the same:
$$
E^{*}(\vec \tau, s, \tilde L, \bold 1) = E^*(\vec\tau, s, \tilde L^+, \bold 1).
$$
It is interesting to compare this with Lemma \ref{lem6.4}(2).
\end{remark}

 Theorem \ref{theo1.2} is almost clear; we restate it as follows.

\begin{theorem}  \label{theo6.10} Assume $d_E=D^2 \tilde D$ with $D\equiv 1 \mod 4$ prime and $\tilde D \equiv 1 \mod 4$ square free. Let $\tilde L=\OO_{\tilde E}$  with quadratic form $\tilde Q(z) =-\frac{1}{\sqrt{\tilde D}} z \bar z$. Then
$$
\Phi(\CM(E), f) =c'(E) \left( \CT[\langle f^+(\tau), \mathcal
E(\tau, \tilde L)\rangle]- \LL'(0, \xi(f), \tilde L) \right).
$$
where
$$
c'(E) =\frac{\deg (\CM(E))}{2 \Lambda(0, \chi)}.
$$
In  particular, when $\tilde D$ is also prime, $c'(E) =1$.
\end{theorem}
\begin{proof}
By Proposition  \ref{prop:onelattice}, one has
  $$
  \mathcal E(\tau, \tilde L(\bA)) = \mathcal E(\tau, \tilde L),  \quad \LL(s, \xi(f), \tilde L(\bA))= \LL(s, \xi(f), \tilde L)
  $$
  for all $\bold A \in \CM^\Sigma(E)$.
So   Theorem  \ref{theo6.7} gives
$$
\Phi(\CM(E), f) = c(E) |C(T)\backslash \CM^\Sigma(E)| \left(
\CT[\langle f^+(\tau), \mathcal E(\tau, \tilde L)\rangle]-
\LL'(0, \xi(f), \tilde L) \right).
$$
Finally,
\begin{align*}
c(E)\, |C(T)\backslash \CM^\Sigma(E)|
 &= \frac{4}{w_E} \frac{|\CM^\Sigma(E)|}{\Lambda(0, \chi)}
 \\
 &= \frac{1}{w_E} \frac{|\CM(E)|}{\Lambda(0, \chi)}
 \\
 &=\frac{\deg \CM(E)}{2 \Lambda(0, \chi)}
\end{align*}
as claimed.
The last claim, $c'(E)=1$, follows from this and \cite[(9.2)]{BY1},
since our points here have multiplicity $2/w_E$ and our $\CM(E)$ is twice the CM cycle there.
\end{proof}

\subsection{Scalar modular forms}

In this subsection, we again
assume that $d_E=D^2 \tilde D$ with  $D\equiv 1 \mod 4$ prime and $\tilde D \equiv 1 \mod 4$
square-free, and    translate
Theorem \ref{theo6.10} into the usual language of scalar modular
forms, finally compare it with \cite[Theorem 1.4]{BY1} in the special
case considered there. Under the identification
$$
X \cong \SL_2(\OO_F) \backslash \mathbb H^2,
$$
 the Hirzebruch Zagier divisor $T_n$ defined in \cite{BY1} is related to the special divisor $Z(m, \mu)$
via
\begin{equation} \label{neweq6.17}
T_n = \frac{1}2
\begin{cases}
Z(\frac{n}D, 0) &\ff D|n,
\\
Z(\frac{n}D,  \mu) + Z(\frac{n}D, -\mu) &\ff D\nmid n.
\end{cases}
\end{equation}
Here, in the second case, $\mu\in \tilde L'/\tilde L$ is determined by the condition that $Q(\mu) \equiv \frac{n}D \mod 1$.
Let $k$ be an even integer, and let
 $A_{k, \rho}$ be the space of
real analytic modular forms of weight $k$ with representation $\rho$,
where
$\rho=\rho_{L}$ or $\bar{\rho}_{L}$. Let $A_k^+(D, (\frac{D}{\cdot}))$ be
the space of real analytic modular forms $f_\scal(\tau)  =\sum_n a(n, v)
q^n$ of weight $k$ for the group $\Gamma_0(D)$ with character $(\frac{D}{\cdot})$ such that
$ a(n, v) =0 $ whenever $(\frac{D}{n})=-1$.  Here we use $f_\scal$ to
denote a scalar valued modular form to distinguish it from vector valued modular forms in this paper. Then the following lemma is proved in
\cite{BB}.

\begin{lemma}
\label{lem:iso}
There is an isomorphism of vector spaces $A_{k, \rho} \to  A_k^+(D, (\tfrac{D}{\cdot}))$,
\begin{align*}
f=\sum_{\mu \in L'/L} f_\mu \ph_\mu &\mapsto  f_\scal=
D^{\frac{1-k}2} f_0|W_D.
\end{align*}
The inverse map is given by
$$
f_\scal  \mapsto f =\frac{1}{2}D^{\frac{k-1}2} \sum_{\gamma \in \Gamma_0(D)
\backslash \SL_2(\mathbb Z)}  \left(f_\scal |W_D |\gamma\right) \rho(\gamma)^{-1}\ph_0,
$$
where $W_D=\kzxz {0} {-1} {D} {0}$ denotes the Fricke involution.
Moreover,
if $f_\scal(\tau)  =\sum_n a(n, v) q^n$, then $f$ has the Fourier expansion
\[
f=\frac{1}{2}\sum_{\mu \in L'/L} \sum_{\substack{n\in \Z\\ n\equiv DQ(\mu)\;(D)}} \tilde a(n,v)\, q^{n/D}\, \phi_\mu,
\]
where $\tilde a(n,v)= a(n,v)$ if $n\not\equiv 0\mod{D}$, and  $\tilde a(n,v)= 2\, a(n,v)$ if $n\equiv 0\mod{D}$.
\end{lemma}


In particular, the constant term of $f_\scal$ agrees with the constant term of $f$ in the $\ph_0$ component.
%
The isomorphisms of Lemma \ref{lem:iso} take harmonic weak Maass forms to
harmonic weak Maass forms, (weakly) holomorphic modular forms to  (weakly)
holomorphic modular forms, and cusp forms to cusp forms.

Let
\begin{equation}
E_\scal^{*}(\tau, s) = \frac{1}{\sqrt D} \,E^*(\tau^\Delta, s,  \phi_0, \bold 1)| {W_D}
\end{equation}
be the scalar image of $E^*(\tau^\Delta, s, \tilde L, \bold 1)$, and let
$\mathcal E_\scal(\tau)$ be the holomorphic part of
$$
\tilde f(\tau) = \frac{d}{ds}E_\scal^{*}(\tau,s)|_{s=0}.
$$
Then $\tilde f(\tau)$ is the function defined in \cite[(7.2)]{BY1}.  By \cite[Theorem 7.2]{BY1}, we have the following lemma.

\begin{lemma}
Let the notation be as above. Then
$$
\mathcal E_\scal(\tau) = -2 \Lambda'(0, \chi) - 4 \sum_{m \in \Z_{>0}} b_m q^m
$$
where
$$
b_m= \sum_{\substack{ t= \frac{n+m \sqrt{\tilde D}}{2D}  \in d_{\tilde E/\tilde F}^{-1} \\   |n| < m \sqrt{\tilde D}}} B_t
$$
with
$$
B_t = (\ord_\mathfrak l +1) \rho(t d_{\tilde E/\tilde F} \mathfrak l^{-1}) \log \norm(\mathfrak l)
$$
for some (and any) prime ideal of $\tilde F$ with  $\tilde\chi_\mathfrak l(t) =-1$. Here $\tilde\chi$ is the quadratic Hecke character of $\tilde F$ associated to $\tilde E/\tilde F$. Finally
$$
\rho(\mathfrak a) =|\{ \mathfrak A \subset \OO_{\tilde E} :\; \norm_{\tilde E/\tilde F} \mathfrak A =\mathfrak a\}|.
$$
\end{lemma}

Now let $f_\scal =f_\scal^{+} + f_\scal^{-} \in H_0^+(D, (\frac{D}{\cdot}))$ be a harmonic weak Maass form with holomorphic part
 $$
 f_\scal^{+}(\tau) = \sum_{n \gg -\infty} c^+(n) q^n,
 $$
and let
$$
\tilde c^+(n)= \begin{cases}
  2 c^+(n) &\ff D|n,
  \\
  c^+(n) &\ff D \nmid n.
  \end{cases}
$$
Let $f \in H_{0, \bar{\rho}_{{\tilde L}}}$ be the associated vector valued harmonic weak Maass form.
Define
\begin{equation}
T(f_\scal) = \sum_{n >0} \tilde c^+(-n) T_n, \quad  \Phi(z, f_\scal) := \Phi(z, f).
\end{equation}
Then  one sees that $T(f_\scal) = Z(f)$ by  (\ref{neweq6.17}).
Define the Rankin-Selberg $L$-series
\begin{equation}
\LL_\scal(s, \xi(f_\scal),  \tilde L) =\langle  E_\scal^{*}(\tau,s), \xi(f_\scal)
\rangle_{\Pet}.
\end{equation}

Then a straightforward calculation gives

\begin{lemma}
 (1) \quad
$$
\LL(s, \xi(f), \tilde L) =\frac{1}2 D(D+1) \LL_\scal(s, \xi(f_\scal),\tilde L) .
$$

(2) \quad
$$
\CT[\langle f^+, \mathcal E(\tau, \tilde L)\rangle]
  = -2 c^+(0) \Lambda'(0,\chi) - 2 \sum_{n>0} \tilde c^+(-n) b_n.
$$
\end{lemma}

Combining this with Theorem \ref{theo6.10}, we obtain:

\begin{corollary}
\label{cor6.14}
Let $F=\Q(\sqrt D)$ with $D \equiv 1 \mod 4$ prime, and let $E$ be a CM  non-biquadratic field  with absolute discriminant $d_E =D^2 \tilde D$ where  $\tilde D \equiv 1 \mod 4$ is square free.  If $f_\scal \in H_0^+(D, (\frac{D}{\cdot}))$, then
$$
\Phi(\CM(E), f_\scal) =-2 c'(E) \left[ \sum_{n>0} \tilde c^+(-n) b_n +
c^+(0) \Lambda'(0, \chi) + \frac{D(D+1)}4 \LL_\scal'(0, \xi(f_\scal), \tilde L)\right].
$$
\end{corollary}

Now we assume that $f_\scal =\sum c^+(n) q^n \in H_0^+(D,
(\frac{D}{\cdot})) $ is weakly holomorphic, i.e., $\xi(f_\scal)=0$, and
that $\tilde c^+(n)\in \Z$ for $n<0$.  Then there is a (up to a
constant of modulus $1$ unique) memomorphic Hilbert modular form
$\Psi(z, f_\scal)$  of weight $c^+(0)$ with a Borcherds product
expansion whose divisor is given by
$$
\dv (\Psi) = T(f_\scal),
$$
see \cite[Theorem 9]{BB}. Morever,  by construction it satisfies
$$
-\log \|\Psi(z, f_\scal)\|_\Pet^2 =\Phi(z, f_\scal),
$$
where
$$
\|\Psi(z_1, z_2, f_\scal)\|_\Pet^2 =|\Psi(z_1, z_2, f_\scal)|^2 (4 \pi
e^{-\gamma}y_1 y_2)^{c^+(0)}
$$
is the Petersson metric (normalized in a way which is convenient for
our purposes), and $\gamma= -\Gamma'(1)$ is Euler's constant.

\begin{corollary} \label{cor6.13}
Let the notation be as in  Corollary \ref{cor6.14} and assume that
$f_\scal$ is weakly homomorphic. Then
$$
\log \| \Psi(\CM(E), f_\scal)\|_\Pet =c'(E)  \sum_{n>0} \tilde c^+(-n)
b_n + c'(E)  c^+(0) \Lambda'(0, \chi) .
$$
\end{corollary}

When $\tilde D$ is also prime, we have $c'(E)=1$. Then this
corollary coincides with \cite[Theorem 1.4]{BY1}, since the CM
points in this paper are counted with multiplicity $\frac{2}{w_E}$,
and our CM cycle is twice the CM cycle there as a set with multiplicities.
%

 Combining this corollary with the \cite[Theorem 1.2]{YaGeneral}, one has the following theorem,
 which verifies a special case of Conjecture \ref{conj5.3}.

\begin{theorem}  \label{theo6.16} Assume that $E$ is a quartic CM
number field with absolute discriminant $d_E= D^2 \tilde D$ and real quadratic subfield $F=\Q(\sqrt D)$
such that $D \equiv 1 \mod 4$ is prime and $\tilde D \equiv 1 \mod 4$ is square-free. Assume further that
$$
\OO_E=\OO_F + \OO_F \frac{w+\sqrt\Delta}2
$$
is free over $\OO_F$, where $w, \Delta \in \OO_F$. Then  $c'(E) =1$.
Moreover, let $\mathcal X$ be a regular toroidal compactification of the
moduli stack of principally polarized abelian surfaces with real
multiplication by $\OO_F$, \cite{rapoport.HB}, \cite{deligne-pappas}, and
 let $\mathcal T_n$ be the closure of $T_n$ in  $\mathcal X$. Let $\CCM(E)$ be the moduli stack of the principally polarized
abelian surfaces with CM by $\OO_E$.
Then for any
$f_\scal \in H_0^+(D, (\frac{D}{\cdot}))$, one has
$$
\langle \hat{ \mathcal T}(f_\scal),  \CCM(E) \rangle_{\text{\rm Fal}} =-
\frac{1}2 c^+(0) \Lambda'(0, \chi) -\frac{D(D+1)}8  \LL_\scal'(0, \xi(f_\scal),
\tilde L).
$$
Here
$$\hat{\mathcal  T}(f_\scal) =(\mathcal T(f_\scal), \Phi(z, f_\scal) )=\big(\,\sum_{n>0} \tilde c^+(-n) \mathcal T_n, \Phi(z, f_\scal)\,\big)\in
 \widehat{\CH}^1(\mathcal X)_\C.$$
\end{theorem}

\begin{proof}(sketch)
First notice that $\CCM(E)$ does not meet with
  the boundary of $\mathcal X$ and that $\CCM(E)$ intersects with $\mathcal
  T(f_\scal)$ properly. Notice also that $\CCM(E)(\mathbb C) = \frac{1}2
  \CM(E)$ since each point in $\CCM(E)(\mathbb C)$ is counted with
  multiplicity $\frac{1}{w_E}$ (each point $\bA$ in $\mathcal X$ is
  counted with multiplicity $\frac{1}{\Aut(\bA)}$).

First, take $f_\scal$ to be non-trivial weakly holomorphic with
$c^+(0)=0$ so that
 $\hat{\calT}(f_\scal) =0$ in $\widehat{\CH}^1(\mathcal X)_\C$. So
$$
0= \langle \hat{\mathcal T}(f_\scal),  \CCM(E) \rangle_{\text{Fal}} =
\langle \mathcal T(f_\scal),  \CCM(E) \rangle_{\text{fin}} - \frac{1}2
\log|\Psi(\CM(E), f_\scal)|,
$$
and consequently
$$
\langle \mathcal T(f_\scal),  \CCM(E) \rangle_{\text{fin}}=\frac{1}2
c'(E) \sum_{n>0} \tilde c^+(-n) b_n
$$
by Corollary \ref{cor6.13}. On the other hand, \cite[Theorem
1.2]{YaGeneral} (which uses Corollary \ref{cor6.13}) asserts that
\begin{equation} \label{eq6.22}
\langle \mathcal T(f_\scal),  \CCM(E) \rangle_{\text{fin}}=\frac{1}2
 \sum_{n>0} \tilde c^+(-n) b_n.
\end{equation}
So $c'(E)=1$.

Now for a general $f_\scal$, one has by definition, (\ref{eq6.22}) and
 Corollary \ref{cor6.14} that
\begin{align*}
\langle \hat{\mathcal T}(f_\scal),  \CCM(E) \rangle_{\text{Fal}}
 &=\langle \mathcal T(f_\scal),  \CCM(E) \rangle_{\text{fin}} -\frac{1}4
 \Phi(\CM(E), f_\scal)
 \\
  &=-
\frac{1}2 c^+(0) \Lambda'(0, \chi) -\frac{D(D+1)}8  \LL_\scal'(0, \xi(f_\scal),
\tilde L).
\end{align*}
This concludes the proof of the theorem.
\end{proof}

 We remark that (\ref{eq6.22}) verifies a special
 case of Conjecture \ref{conj5.2} and that this corollary is a generalization of \cite[Theorem 1.3]{YaGeneral}. A slight refinement of the main
 result in \cite{HY} together with Theorem \ref{theo6.7} should
 settle Conjectures \ref{conj5.2} and \ref{conj5.3} completely for
 Hilbert modular surfaces. It would also settle the Colmez
 conjecture (\cite{Col}, \cite{YaColmez})  for  CM abelian surfaces.

\end{document}